\documentclass[12pt,oneside,reqno]{amsart}

\usepackage[matrix,arrow,curve,frame]{xy}
\usepackage{amsmath,amsthm,amssymb,enumerate}
\usepackage{latexsym}
\usepackage{amscd}
\usepackage[colorlinks=false]{hyperref}
\usepackage{euscript}

\newtheorem{thm}[subsection]{Theorem}

\newtheorem{cor}[subsection]{Corollary}
\newtheorem{lemma}[subsection]{Lemma}

\newtheorem{remark}[subsection]{Remark}

\theoremstyle{definition}

\numberwithin{equation}{section}

\def\phi{{\varphi}}

\def\cA{{\mathcal A}}
\def\cB{{\mathcal B}}
\def\cC{{\mathcal C}}
\def\cD{{\mathcal D}}

\def\cF{{\mathcal F}}

\def\cH{{\mathcal H}}
\def\cI{{\mathcal I}}

\def\cN{{\mathcal N}}
\def\cO{{\mathcal O}}

\def\cS{{\mathcal S}}

\def\cV{{\mathcal V}}
\def\cW{{\mathcal W}}

\def\gg{{\mathfrak g}}

\def\gl{{\mathfrak l}}

\def\go{{\mathfrak o}}
\def\gp{{\mathfrak p}}

\def\gs{{\mathfrak s}}

\newfont{\german}{eufm10}

\begin{document}

\pagestyle{plain}

\title
{Orbifolds of Gaiotto-Rap\v{c}\'ak $Y$-algebras}

\author{Masoumah Al-Ali}
\address{Saudi Electronic University, Riyadh, Saudia Arabia}
\email{m.alali@seu.edu.sa}

\author{Andrew R. Linshaw}
\address{University of Denver}
\email{andrew.linshaw@du.edu}

\thanks{A. L. is supported by Simons Foundation Grant \#635650 and NSF Grant DMS-2001484.}

\begin{abstract} 
The universal two-parameter $\cW_{\infty}$-algebra is a classifying object for vertex algebras of type $\cW(2,3,\dots, N)$ for some $N$. Gaiotto and Rap\v{c}\'ak recently introduced a large family of such vertex algebras called $Y$-algebras, which includes many known examples such as the principal $\cW$-algebras of type $A$. These algebras admit an action of $\mathbb{Z}_2$, and in this paper we study the structure of their orbifolds. Aside from the extremal cases of either the Virasoro algebra or the $\cW_3$-algebra, we show that these orbifolds are generated by a single field in conformal weight $4$, and we give strong finite generating sets. \end{abstract}

\keywords{Vertex algebra, universal $\cW_{\infty}$-algebra, $\cW$-algebra, orbifold}
\maketitle

\section{Introduction}
Gaiotto and Rap\v{c}\'ak introduced an important family of vertex algebras $Y_{N_1, N_2, N_3}[\psi]$ called  $Y$-algebras, which are indexed by three integers $N_1, N_2, N_3\geq 0$ and a complex parameter $\psi$ \cite{GR}. These vertex algebras are associated to interfaces of twisted $\cN=4$ supersymmetric gauge theories with gauge groups $U(N_1)$,  $U(N_2)$, and $U(N_3)$. The interfaces satisfy a permutation symmetry which is expected to induce a corresponding symmetry on the vertex algebras. This led Gaiotto and Rap\v{c}\'ak to conjecture a triality of isomorphisms of $Y$-algebras. Define $\psi = -\frac{\epsilon_2}{\epsilon_1}$ where $\epsilon_1 + \epsilon_2 +\epsilon_3 =0$,
and set 
$$ Y^{\epsilon_1, \epsilon_2, \epsilon_3}_{N_1, N_2, N_3} := Y_{N_1, N_2, N_3}[\psi].$$
In this, notation the triality conjecture is
\begin{equation} \label{GRconj}
Y^{\epsilon_{\sigma(1)}, \epsilon_{\sigma(2)}, \epsilon_{\sigma(3)}}_{N_{\sigma(1)}, N_{\sigma(2)}, N_{\sigma(3)}} \cong Y^{\epsilon_1, \epsilon_2, \epsilon_3}_{N_1, N_2, N_3},\  \text{for} \ \sigma \in S_3.\end{equation}
In the case when one of the labels $N_1, N_2, N_3$ is zero, $Y_{N_1, N_2, N_3}[\psi]$ is a simple vertex algebra, and it can be realized (up to a Heisenberg algebra) as a certain coset of an affine $\cW$-(super)algebra. In this case, the triality conjecture is equivalent to the statement that there are three different such realizations, and this was proven by Creutzig and the second author in \cite{CLIII}. 

More precisely, for $n\geq 1$, consider the $\cW$-algebras
$$\cW^k(\gs\gl_{n+m}, f_{n+m}),\qquad \psi = k+n+m,$$ where $f_{n+m}$ is the nilpotent element which is principal in $\gs\gl_n$ and trivial in $\gs\gl_m$. Note that $\cW^k(\gs\gl_{n+m}, f_{n,m})$ is a common generalization of the principal $\cW$-algebra $\cW^k(\gs\gl_n)$ (the case $m=0$), the affine vertex algebra $V^k(\gs\gl_{m+1})$ (the case $n=1$), the subregular $\cW$-algebra $\cW^k(\gs\gl_{n+1}, f_{\text{subreg}})$ (the case $m=1$), and the minimal $\cW$-algebra $\cW^k(\gs\gl_{m+2}, f_{\text{min}})$ (the case $n=2$). As in \cite{CLIII}, we replace $k$ with the shifted level $\psi = k+n+m$, and we set
$$\cW^{\psi}(n,m) := \cW^k(\mathfrak{sl}_{n+m}, f_{n+m}) =  \cW^{\psi - n-m}(\mathfrak{sl}_{n+m}, f_{n+m}).$$
For $m\geq 1$, $\cW^{\psi}(n,m)$ has affine vertex subalgebra
$$  \left\{
\begin{array}{ll}
V^{-\psi-m-1}(\gg\gl_m) = \cH \otimes V^{-\psi-m-1}(\gs\gl_m) & m\geq 2,
\smallskip
\\ \cH & m=1, \\
\end{array} 
\right.
$$ where $\cH$ is the Heisenberg algebra.
For $n=0$, a different definition of $\cW^{\psi}(0,m)$ was given \cite{CLIII}, and this is omitted for brevity. Consider the affine cosets
$$\cC^{\psi}(n, m) =   \left\{
\begin{array}{lll}
 \text{Com}(V^{\psi - m-1}(\mathfrak{gl}_m), \cW^{\psi}(n, m)) &  & m\geq 2, \ n\geq 0,
\smallskip
\\  \text{Com}(\cH, \cW^{\psi}(n, 1)) &  & m=1, \ n\geq 0,
\smallskip
\\ \cW^{\psi}(n, 0) &  & m = 0,\ n\geq 2.
\end{array} 
\right.$$
Up to a Heisenberg algebra, these cosets are just the above $Y$-algebras, that is,
\begin{equation}
\begin{split} 
Y_{0, M, N}[\psi] &= \cC^\psi(N-M, M) \otimes \cH , \qquad M \leq N, \\
Y_{0, M, N}[\psi] &= \cC^{-\psi+1}(M-N, N)  \otimes \cH, \qquad M > N. \\
\end{split} \end{equation}

Next, for $n\geq 1$ and $n\geq m$, there is a similar family of affine $\cW$-superalgebras 
$$\cW^k(\gs\gl_{n|m}, f_{n|m}),\qquad \psi = k+n+m,$$ where $f_{n|m}$ is the nilpotent element which is principal in $\gs\gl_n$ and trivial in $\gs\gl_m$. As above, $\cW^k(\gs\gl_{n|m}, f_{n|m})$ is a common generalization of the principal $\cW$-algebra $\cW^k(\gs\gl_n)$ (the case $m=0$), the affine vertex superalgebra $V^k(\gs\gl_{1|m})$ (the case $n=1$), the principal $\cW$-superalgebra $\cW^k(\gs\gl_{n|1}, f_{\text{subreg}})$ (the case $m=1$), and the minimal $\cW$-superalgebra $\cW^k(\gs\gl_{2|m}, f_{\text{min}})$ (the case $n=2$). 
In the case $n=m$, we need to take $\cW^k(\gp\gs\gl_{n|n}, f_{n|n})$ instead of $\cW^k(\gs\gl_{n|n}, f_{n|n})$ to get a simple algebra. Again, we replace the level $k$ with the shifted level $\psi = k+n-m$, and set
$$\cV^{\psi}(n,m) =  \left\{
\begin{array}{ll}
\cW^k(\gs\gl_{n|m}, f_{n|m}) & n\geq 1,\quad m \neq n,
\smallskip
\\ \cW^k(\gp\gs\gl_{n|n}, f_{n|n}) & n\geq 2,  \quad m = n. \\
\end{array} 
\right.
$$
For $m\geq 1$, $\cV^{\psi}(n,m)$ has affine vertex subalgebra
$$  \left\{
\begin{array}{ll}
V^{-\psi-m+1}(\gg\gl_m) = \cH \otimes V^{-\psi-m+1}(\gs\gl_m) & m\geq 2, \quad m \neq n,
\smallskip
\\ V^{-\psi-n+1}(\gs\gl_n)  & n\geq 2, \quad m = n,
\\ \cH & m=1. \\
\end{array} 
\right.
$$
The extremal cases $\cV^{\psi}(1,1)$, $\cV^{\psi}(1,0)$, and $\cV^{\psi}(0,m)$ are defined differently in \cite{CLIII}, and this is omitted for brevity. Define
\begin{equation*} \begin{split} &  \cD^{\psi}(n,m) = \left\{
\begin{array}{ll}
\text{Com}(V^{-\psi-m+1}(\gg\gl_m), \cV^{\psi}(n,m)) & n\geq 0, \quad m\geq 2, \quad m\neq n ,
\smallskip 
\\ \text{Com}(V^{-\psi-n+1}(\gs\gl_n), \cV^{\psi}(n,n))^{GL_1}  & n\geq 2, \quad m = n ,
\smallskip 
\\ \text{Com}(\cH, \cV^{\psi}(n,1)) & n \geq 2, \quad m=1,
\\  \cV^{\psi}(n,0)& n\geq 2, \quad m = 0.
\end{array} 
\right.
\end{split} \end{equation*}
Again, we omit the definition of the extremal cases $\cD^{\psi}(1,1)$, $\cD^{\psi}(1,0)$, and $\cD^{\psi}(0,1)$, which can be found in \cite{CLIII}. Note that $\cD^{\psi}(n,m)$ is just the affine coset for $n\neq m$, but in the case $n=m$ there is an action of $GL_1$ by outer automorphisms, and it is necessary to take the $GL_1$-orbifold of the affine coset. Up to a Heisenberg algebra, these cosets are just the above $Y$-algebras:
\begin{equation}
\begin{split} 
Y_{L, 0, N}[\psi] &= \cD^\psi(N, L) \otimes \cH, \\
Y_{L, M, 0}[\psi] &= \cD^{-\psi+1}(M, L) \otimes \cH.
\end{split}
\end{equation}

The main result of \cite{CLIII}, which is equivalent to the Gaiotto-Rap\v{c}\'ak triality conjecture in the case where one of the labels is zero, is the following.
\begin{thm} \label{CLIIImain} Let $n\geq  m \geq 0$ be non-negative integers. As one-parameter vertex algebras
$$\cD^\psi(n, m)  \cong \cC^{\psi^{-1}}(n-m, m) \cong \cD^{\psi'}(m, n), \qquad \psi' = \frac{1}{\psi} +\frac{1}{\psi'} =1.$$
\end{thm}

This theorem is a vast generalization of well-known results such as Feigin-Frenkel duality \cite{FFII} and the coset realization of principal $\cW$-algebras of type $A$ \cite{ACL}, which are the special cases $\cD^\psi(n, 0) \cong \cC^{\psi^{-1}}(n, 0)$ and $\cD^{\psi}(n, 0) \cong \cD^{\psi'}(0, n)$, respectively. The $Y$-algebras are important because they serve as building blocks for many interesting vertex algebras. Conjecturally, they are isomorphic to the $W_{N_1, N_2, N_3}$-algebras of \cite{BFM}, which act on the moduli space of spiked instantons of certain toric Calabi-Yau threefolds \cite{RSYZ}. In these example, the toric diagram has three two-dimensional faces, and each face is labelled by integers $N_1, N_2, N_3$ that indicate an action of the gauge groups $U(N_1)$, $U(N_2)$ and $U(N_3)$. As explained in the introduction of \cite{CLIII}, when one of these labels is zero, we expect that the $Y_{0, M, N}$-algebra has representation categories which are Kazhdan-Lusztig categories $KL_k(\gs\gl_M)$ and $KL_\ell(\gs\gl_N)$ for some $k, \ell$. The extension of a tensor product of two $Y$-algebras along a common $KL_k(\gs\gl_M)$ should correspond to a toric Calabi-Yau threefold whose toric diagram has four faces, and this procedure can be iterated to create more complicated vertex algebras.

For the rest of this paper, we will use the cosets $\cC^{\psi}(n,m)$ as our realization of the $Y$-algebras. One of the key ingredients in the proof of Theorem \ref{CLIIImain} is that aside from the extremal cases $\cC^{\psi}(0,0), \cC^{\psi}(0,1), \cC^{\psi}(1,0)$ (which are isomorphic to $\mathbb{C}$), and $\cC^{\psi}(0,2), \cC^{\psi}(0,2)$ (which are isomorphic to the Virasoro algebra), $\cC^{\psi}(n,m)$ can be realized explicitly as a quotient of the universal two-parameter $\cW_{\infty}$-algebra $\cW(c,\lambda)$ constructed by the second author in \cite{LII}. 

\subsection{Orbifolds of $Y$-algebras}
The two-parameter vertex algebra $\cW(c,\lambda)$ has full automorphism group $\mathbb{Z}_2$ \cite[Cor. 5.3]{LII}. Aside from the above extremal cases, all the algebras $\cC^{\psi}(n,m)$ inherit this $\mathbb{Z}_2$-symmetry. In this paper, our main goal is to study the structure of the orbifold $\cW(c,\lambda)^{\mathbb{Z}_2}$ as a two-parameter vertex algebra, and $\cC^{\psi}(n,m)^{\mathbb{Z}_2}$ as one-parameter vertex algebras, or equivalently, for generic values of $\psi$. We shall adopt the following notation: if a vertex algebra $\cV$ has a minimal strong generating set consisting of $d_i$ fields in weight $n_i$ for $i = 1,2,\dots, r$, we say that $\cV$ is of type $\cW(n_1^{d_1}, \dots, n_r^{d_r})$.

Certain special cases of this problem have previously been studied in the physics literature \cite{B-H}. For example, it was conjectured that $\cW^k(\gs\gl_3)^{\mathbb{Z}_2}$ should be of type $\cW(2,6,8,10,12)$, and this was proven in our earlier paper \cite{AL}. Similar conjectures were given in \cite{B-H} for $\cW^k(\gs\gl_n)$ for $n = 4,5,6$:
\begin{enumerate}
\item $\cW^k(\mathfrak{sl}_4)^{\mathbb{Z}_2}$ should be of type $\cW(2,4,6,8,10,12)$.
\item $\cW^k(\mathfrak{sl}_5)^{\mathbb{Z}_2}$ should be of type $\cW(2,4,6,8^2,9,10^3,11,12^3,13,14^2)$.
\item $\cW^k(\mathfrak{sl}_6)^{\mathbb{Z}_2}$ should be of type $\cW(2,4,6^2,8^2,9,10^3,11,12^3,13,14^2)$.
\end{enumerate}

Also, the parafermion algebra of $\gs\gl_2$, namely, $N^k(\gs\gl_2) = \text{Com}(\cH, V^k(\gs\gl_2))$, is just $\cC^{\psi}(1,1)$. It was conjectured in \cite{B-H} that $N^k(\gs\gl_2)^{\mathbb{Z}_2}$ should be of type $\cW(2,4,6,8,10)$, and this was proven by Kanade and the second author \cite{KL}.  We mention that Jiang and Wang \cite{JWI,JWII} have classified the irreducible modules and computed the fusion rings for the simple parafermion algebras $K_k(\gs\gl_2)^{\mathbb{Z}_2} = \text{Com}(\cH, L_k(\gs\gl_2))^{\mathbb{Z}_2}$ for $k \in \mathbb{N}$, which are $C_2$-cofinite and rational. Here $L_k(\gs\gl_2)$ denotes the simple affine vertex algebra. Due to the well-known isomorphism $K_k(\gs\gl_2) \cong \cW_r(\gs\gl_k)$ for $r =-k + \frac{k+1}{k+2}$ of Arakawa, Lam, and Yamada \cite{ALY}, the orbifolds $K_k(\gs\gl_2)^{\mathbb{Z}_2}$ are really orbifolds of certain rational $\cW$-algebras of type $A$. One of the motivations for this work is to study the representation theory of $\cW_r(\gs\gl_k)^{\mathbb{Z}_2}$ for a general nondegenerate admissible level $r$, which are $C_2$-cofinite and rational \cite{AI,AII}.

\subsection{Main result} Our main result is that after a suitable localization of the ring $\mathbb{C}[c,\lambda]$, $\cW(c,\lambda)^{\mathbb{Z}_2}$ is generated by the weight $4$ field $W^4$ as a two-parameter vertex algebra. A consequence is that after a suitable localization of the ring $\mathbb{C}[\psi]$, all the one-parameter quotients $\cC^{\psi}(n,m)^{\mathbb{Z}_2}$ for $n\geq 4$ and $m=0$, and for $n,m\geq 1$, are generated by $W^4$. In particular, this holds for $\cW^k(\gs\gl_n)^{\mathbb{Z}_2}$ for all $n\geq 4$. We also give a finite strong generating set for $\cC^{\psi}(n,m)^{\mathbb{Z}_2}$. Our strong generating set is not typically minimal, although it is minimal for $\cW^k(\gs\gl_n)^{\mathbb{Z}_2}$ for $n = 4,5,6$ and the above conjectures from \cite{B-H} are true. We will also give the minimal strong generating set for $\cW^k(\gs\gl_7)^{\mathbb{Z}_2}$, which illustrates the subtlety of this problem in general. Our strong generating set yields a finite generating set for Zhu's algebra $A(\cC^{\psi}(n,m)^{\mathbb{Z}_2})$ \cite{Z}. The description of this algebra is an important step in understanding the representation theory of these algebras. The strategy of proof is as follows.
\begin{enumerate}
\item First, we consider the special case $\cC^{\psi}(n,0) = \cW^k(\gs\gl_n)$ for $n \geq 4$, which has the advantage that unlike $\cC^{\psi}(n,m)$ for $m\geq 1$, it admits a large level limit $\cW^{\text{free}}(\gs\gl_n) = \lim_{k\rightarrow \infty} \cW^k(\gs\gl_n)$ which is a free field algebra. We find both minimal weak and strong finite generating sets for $\cW^{\text{free}}(\gs\gl_n)^{\mathbb{Z}_2}$. These give rise to weak and strong finite generating sets for $\cW^{k}(\gs\gl_n)^{\mathbb{Z}_2}$, respectively, which are no longer minimal.

\item Using the weak and strong generating sets for $\cW^{\text{free}}(\gs\gl_n)^{\mathbb{Z}_2}$ for all $n$, we find economical (but not minimal) infinite weak and strong generating sets for $\cW(c,\lambda)^{\mathbb{Z}_2}$. These yield weak and strong finite generating sets for $\cC^{\psi}(n,m)^{\mathbb{Z}_2}$ for all $n,m$.

\item Using the partial OPE algebra of $\cW(c,\lambda)$ which appears in \cite{LII}, we show that all weak generators for $\cW(c,\lambda)^{\mathbb{Z}_2}$ from (2), and hence the entire algebra, can be generated by the field $W^4$.
\end{enumerate}

\section{Vertex algebras} \label{sec:VOA} We will assume that the reader is familiar with vertex algebras, and we use the same notation as the previous paper \cite{CLIII} of the second author. In this section, we briefly recall the definition and basic properties of free field algebras, $\cW$-algebras, and the universal two-parameter algebra $\cW(c,\lambda)$.

\subsection{Free field algebras} 
A free field algebra is a vertex superalgebra $\cV$ with weight grading
$$\cV = \bigoplus_{d \in \frac{1}{2} \mathbb{Z}_{\geq 0} }\cV[d],\qquad \cV[0] \cong \mathbb{C},$$ with strong generators $\{X^i|\ i \in I\}$ satisfying OPE relations
$$ X^i(z) X^j(w) \sim a_{i,j} (z-w)^{-\text{wt}(X^i) - \text{wt}(X^j)},\quad a_{i,j} \in \mathbb{C},\ \quad a_{i,j} = 0\ \ \text{if} \ \text{wt}(X^i) +\text{wt}(X^j)\notin \mathbb{Z}.$$
We now recall the four families of free field algebras that were introduced in \cite{CLIII}.

\smallskip

\noindent {\it Even algebras of orthogonal type}. For each $n\geq 1$ and even $k \geq 2$, $\cO_{\text{ev}}(n,k)$ is the vertex algebra with even generators $a^1,\dots, a^n$ of weight $\frac{k}{2}$, which satisfy
$$a^i(z) a^j(w) \sim \delta_{i,j} (z-w)^{-k}.$$ 
In the case $k=2$, $\cO_{\text{ev}}(n,k)$ just the rank $n$ Heisenberg algebra $\cH(n)$. It has no conformal vector for $k>2$, but for all $k$ it is a simple vertex algebra and has full automorphism group the orthogonal group $\text{O}(n)$.

\smallskip

\noindent {\it Even algebras of symplectic type}. For each $n\geq 1$ and odd $k \geq 1$,  $\cS_{\text{ev}}(n,k)$ is the vertex algebra with even generators $a^i, b^i$ for $i=1,\dots, n$ of weight $\frac{k}{2}$, which satisfy 
\begin{equation} \begin{split} a^i(z) b^{j}(w) &\sim \delta_{i,j} (z-w)^{-k},\qquad b^{i}(z)a^j(w)\sim -\delta_{i,j} (z-w)^{-k},\\ a^i(z)a^j(w) &\sim 0,\qquad\qquad\qquad \ \ \ \ b^i(z)b^j (w)\sim 0.\end{split} \end{equation} In the case $k=1$, $\cS_{\text{ev}}(n,k)$ is just the rank $n$ $\beta\gamma$-system $\cS(n)$. For $k>1$, $\cS_{\text{ev}}(n,k)$ has no conformal vector, but for all $k$ it is simple and has full automorphism group the symplectic group $\text{Sp}(2n)$.

\smallskip

\noindent {\it Odd algebras of symplectic type}. For each $n\geq 1$ and even $k \geq 2$, $\cS_{\text{odd}}(n,k)$ is the vertex superalgebra with odd generators $a^i, b^i$ for $i=1,\dots, n$ of weight $\frac{k}{2}$, which satisfy
\begin{equation} \begin{split} a^{i} (z) b^{j}(w)&\sim \delta_{i,j} (z-w)^{-k},\qquad b^{j}(z) a^{i}(w)\sim - \delta_{i,j} (z-w)^{-k},\\ a^{i} (z) a^{j} (w)&\sim 0,\qquad\qquad\qquad\ \ \ \  b^{i} (z) b^{j} (w)\sim 0. \end{split} \end{equation} In the case $k=2$, $\cS_{\text{odd}}(n,k)$ is just the rank $n$ symplectic fermion algebra $\cA(n)$. It has no conformal vector for $k>2$, but for all $k$ it is simple and has full automorphism group $\text{Sp}(2n)$.

\smallskip

\noindent {\it Odd algebras of orthogonal type}. For each $n\geq 1$ and odd $k \geq 1$, we define $\cO_{\text{odd}}(n,k)$ to be the vertex superalgebra with odd generators $a^i$ for $i=1,\dots, n$ of weight $\frac{k}{2}$, satisfying
\begin{equation}a^i(z) a^j(w) \sim \delta_{i,j} (z-w)^{-k}.\end{equation} For $k=1$, $\cO_{\text{odd}}(n,k)$ is just the free fermion algebra $\cF(n)$. As above, $\cO_{\text{odd}}(n,k)$ has no conformal vector for $k>1$, but it is simple and has full automorphism group $\text{O}(n)$.

\subsection{$\cW$-algebras}\label{sec:W-algebras} 
Let $\gg$ be a simple, finite-dimensional Lie (super)algebra equipped with a nondegenerate, invariant (super)symmetric bilinear form $( \ \ | \ \ )$, and let $f$ be a nilpotent element in the even part of $\gg$. Associated to $\gg$ and $f$ and any complex number $k$, is the $\cW$-(super)algebra $\cW^k(\gg,f)$. The definition is due to Kac, Roan, and Wakimoto \cite{KRW}, and it generalizes the definition for $\gg$ a Lie algebra and $f$ a principal nilpotent given by Feigin and Frenkel \cite{FFI}.

Fix a basis $\{q^\alpha\}_{\alpha \in S}$ for $\gg$ which is homogeneous with respect to parity, and define the corresponding structure constants and parity by
$$[q^\alpha, q^\beta] = \sum_{\gamma \in S}f^{\alpha\beta}_\gamma q^\gamma,\qquad {|\alpha|} = \begin{cases} 0 & \ q^\alpha \ \text{even}, \\ 1 & \ q^\alpha \ \text{odd}. \end{cases}$$
The level $k$ affine vertex algebra of $\gg$ associated to the bilinear form $( \ \ | \ \ )$ has strong generators $\{X^\alpha\}_{\alpha \in S}$ satisfying
$$X^\alpha(z)X^\beta(w) \sim k(q^\alpha|q^\beta) (z-w)^{-2} + \sum_{\gamma\in S}f^{\alpha \beta}_\gamma X^\gamma (w)(z-w)^{-1}.$$
We define $X_\alpha$ to be the field corresponding to $q_\alpha$ where $\{q_\alpha\}_{\alpha \in S}$ is the dual basis of $\gg$ with respect to $( \ \ | \ \  )$. Let $f$ be a nilpotent element in the even part of $\gg$, which we complete to an $\gs\gl_2$-triple $\{f, x, e\} \subseteq \gg$ satisfying 
$$[x, e] =e, \qquad [x, f]=-f, \qquad [e, f]= 2x.$$ We have the decomposition of $\gg$
$$\gg = \bigoplus_{k \in  \frac{1}{2}\mathbb Z} \gg_k, \qquad \gg_k = \{  a \in \gg |\ [x, a] = ka \}. $$
Write $S = \bigcup_k S_k$ and $S_+ = \bigcup_{k>0} S_k$, where $S_k$ corresponds to a basis of $\gg_k$.

As in \cite{KW}, define the complex $C(\gg, f, k) = V^k(\gg) \otimes F(\gg_+) \otimes F(\gg_{\frac{1}{2}})$, where $F(\gg_+)$ is a free field superalgebra associated to the vector superspace $\gg_+ = \bigoplus_{k \in  \frac{1}{2}\mathbb{Z}_{>0}} \gg_k$, and $F(\gg_{\frac{1}{2}})$ is the neutral vertex superalgebra associated to $\gg_{\frac{1}{2}}$ with bilinear form $\langle a, b \rangle = ( f | [a, b] )$. Then $F(\gg_+)$  is strongly generated by fields $\{\varphi_\alpha, \varphi^\alpha\}_{\alpha \in S_+}$, where $\varphi_\alpha$ and $\varphi^\alpha$ have the opposite parity of $q^\alpha$. The OPEs are
$$\varphi_\alpha(z) \varphi^\beta(w) \sim \delta_{\alpha, \beta}(z-w)^{-1}, \qquad \varphi_\alpha(z) \varphi_\beta(w) \sim 0 \sim \varphi^\alpha(z) \varphi^\beta(w).$$
$F(\gg_{\frac{1}{2}})$  is strongly generated by fields $\{\Phi_\alpha\}_{\alpha \in S_{\frac{1}{2}}}$ and $\Phi^\alpha$ and $q^\alpha$ have the same parity. Their OPEs are
\begin{equation*}
\Phi_\alpha(z) \Phi_\beta(w) \sim \langle q^\alpha , q^\beta \rangle (z-w)^{-1}  \sim (f | [q^\alpha, q^\beta])(z-w)^{-1}.
\end{equation*}
There is a $\mathbb{Z}$-grading on $C(\gg, f, k)$ by charge, and a weight $1$ odd field $d(z)$ of charge $-1$,
\begin{equation}
\begin{split}
d(z) &= \sum_{\alpha \in S_+} (-1)^{|\alpha|} :X^\alpha \varphi^\alpha:  
-\frac{1}{2} \sum_{\alpha, \beta, \gamma \in S_+}  (-1)^{|\alpha||\gamma|} f^{\alpha \beta}_\gamma :\varphi_\gamma \varphi^\alpha \varphi^\beta: + \\
&\quad  \sum_{\alpha \in S_+} (f | q^\alpha) \varphi^\alpha + \sum_{\alpha \in S_{\frac{1}{2}}} :\varphi^\alpha \Phi_\alpha:,
\end{split} 
\end{equation}
whose zero-mode $d_0$ is a square-zero differential on $C(\gg, f, k)$. The $\cW$-algebra $\cW^k(\gg, f)$ is defined to be the homology $H(C(\gg, f, k), d_0)$. 
It has Virasoro element \begin{equation} \begin{split}  L & =  L_{\text{sug}} + \partial x + L_{\text{ch}} + L_{\text{ne}}, \ \text{where} \\ L_{\text{sug}} &=  \frac{1}{2(k+h^\vee)} \sum_{\alpha \in S} (-1)^{|\alpha|} :X_\alpha X^\alpha:,\\ L_{\text{ch}} &= \sum_{\alpha \in S_+} \left(-m_\alpha :\varphi^\alpha \partial \varphi_\alpha: + (1-m_\alpha) :(\partial\varphi^\alpha )\varphi_\alpha:  \right),\\ L_{\text{ne}} &= \frac{1}{2} \sum_{\alpha \in S_{\frac{1}{2}}} :(\partial \Phi^\alpha) \Phi_\alpha : . \end{split} \end{equation} Here $m_\alpha=j$ if $\alpha \in S_j$. The key structural theorem is the following.

\begin{thm} \cite[Thm 4.1]{KW}
Let $\gg$ be a simple finite-dimensional Lie superalgebra with an invariant bilinear
form $( \ \ |  \ \ )$, and let $x, f$ be a pair of even elements of $\gg$ such that ${\rm ad}\ x$ is diagonalizable with
eigenvalues in $\frac{1}{2} \mathbb Z$ and $[x,f] = -f$. Denote by $\gg^f$ the centralizer of $f$ in $\gg$, and suppose that all eigenvalues of ${\rm ad}\ x$ on $\gg^f$ are non-positive, so that
$\gg^f = \bigoplus_{j\leq 0} \gg^f_j$. Then
\begin{enumerate}
\item For each $q^\alpha \in \gg^f_{-j}$, ($j\geq  0$) there exists a $d_0$-closed field $K^\alpha$ of weight $1 + j$, with respect to $L$.
\item The homology classes of the fields $K^\alpha$, where $\{ q^\alpha\}$ is a basis of $\gg^f$, freely generate $\cW^k(\gg, f)$.
\item $H_0(C(\gg, f, k), d_0) = \cW^k(\gg, f)$ and $H_j(C(\gg, f, k), d_0) = 0$ if $j \neq 0$.
\end{enumerate}
\end{thm}

\subsection{Free field limits of $\cW$-algebras} \label{sec:freewalg} 
We recall the notion of a {\it deformable family} of vertex algebras from \cite{LI,CLI,CLII}. It is a vertex algebra $\cB$ defined over the ring $F_K$ of rational functions $\frac{p(\kappa)}{q(\kappa)}$ in a formal variable $\kappa$, where $\text{deg} \ p \leq \text{deg} \ q$, and the zeroes of $q$ lie in some at most countable subset $K \subseteq \mathbb{C}$. Then $\cB^{\infty} = \lim_{\kappa \rightarrow \infty} \cB$ is a well-defined vertex algebra over $\mathbb{C}$. 

Let $\gg$ be a Lie (super)algebra with a nondegenerate form $( \ \ |  \ \ )$, and let $f \in \gg$ be an even nilpotent. It was shown in \cite{CLIII} that after rescaling the generators of $\cW^k(\gg,f)$ by $\frac{1}{\sqrt{k}}$, there exists a deformable family $\cW(\gg,f)$ with parameter $\kappa = \sqrt{k}$ such that $$\cW^k(\gg,f) \cong \cW(\gg,f) / (\kappa - \sqrt{k}) \cW(\gg,f),\ \text{for all} \ k \neq 0.$$ Here $(\kappa - \sqrt{k}) \cW(\gg,f)$ denotes the vertex algebra ideal generated by $\kappa - \sqrt{k}$, regarded as an element of the weight zero subspace. Moreover, its large level limit $\lim_{\kappa \rightarrow \infty} \cW(\gg,f)$, which we denote by $\cW^{\text{free}}(\gg,f)$, has the following property.

\begin{thm} \label{thm:wfreelimit} \cite[Thm. 3.5 and Cor. 3.4]{CLIII} Let $\gg$ be a Lie superalgebra  with invariant, nondegenerate supersymmetric bilinear form, and let $f\in \gg$ be an even nilpotent. Then 
$$\cW^{\text{free}}(\gg,f) \cong \bigotimes_{i=1}^m \cV_i,$$  where each $\cV_i$ is one of the standard free field algebras $\cO_{\text{ev}}(n,k)$, $\cO_{\text{odd}}(n,k)$,  $\cS_{\text{ev}}(n,k)$, or $\cS_{\text{odd}}(n,k)$.
\end{thm}

Recall the decomposition $\gg^f = \bigoplus_{j\geq 0} \gg^f_{-j}$, where $j \in  \frac{1}{2} \mathbb{N}$. Fix a basis $J^f_{-j}$ for $\gg^f_{-j}$, and recall the field $K^{\alpha} \in \cW^k(\gg,f)$ of weight $1+j$ corresponding to $q^{\alpha} \in J^f_{-j}$. We denote the corresponding fields in the limit $\cW^{\text{free}}(\gg,f)$ by $X^{\alpha}$, and we partition $J^f_{-j}$ into subsets $J^f_{-j, \text{ev}}$ and $J^f_{-j, \text{odd}}$ consisting of even and odd elements. Then

\begin{enumerate}
\item If $j$ is a half-integer, $\{X^{\alpha}|\ \alpha \in J^f_{-j,\text{ev}}\}$ has a skew-symmetric pairing, and generates an even algebra of symplectic type.
\item If $j$ is an integer, $\{X^{\alpha}|\ \alpha \in J^f_{-j,\text{ev}}\}$ has a symmetric pairing, and generates an even algebra of orthogonal type.
\item If $j$ is a half-integer, $\{X^{\alpha}|\ \alpha \in J^f_{-j,\text{odd}}\}$ has a symmetric pairing, and generates an odd algebra of orthogonal type.
\item If $j$ is an integer, $\{X^{\alpha}|\ \alpha \in J^f_{-j,\text{odd}}\}$ has a skew-symmetric pairing, and generates an odd algebra of symplectic type.
\end{enumerate}

Theorem \ref{thm:wfreelimit} has several applications, including:
\begin{enumerate}
\item Simplicity of $\cW^{k}(\gg,f)$ for generic values of $k$ whenever $\gg$ admits an invariant, nondegenerate supersymmetric bilinear form \cite[Thm. 3.6]{CLIII}, 
\item Strong finite generation of $\cW^{k}(\gg,f)^G$ for any reductive group of automorphisms $G$, and generic values of $k$ \cite[Thm. 4.1]{CLIII}.
\end{enumerate} 
Statement (2) is based on three observations. First, taking $G$-invariants commutes with taking the limit in the following sense:
\begin{lemma} \cite[Lemma 4.1]{CLIII} \label{lem:orbifold} Let $\gg$ be a Lie (super)algebra with a nondegenerate form, $f \in \gg$ an even nilpotent element, and let $\cW^{\text{free}}(\gg,f) \cong \bigotimes_{i=1}^m \cV_i$ be its free field limit. Let $G$ be a reductive group of automorphisms of $\cW^k(\gg,f)$ as a one-parameter vertex algebra which acts trivially on the ring of rational functions of $k$. Then $\cW(\gg,f)^G$ is a deformable family and $$\lim_{\kappa \rightarrow \infty} \big(\cW(\gg,f)^G\big) \cong  \big(\lim_{\kappa \rightarrow \infty} \cW(\gg,f)\big)^G \cong \big(\bigotimes_{i=1}^m \cV_i\big)^G.$$
\end{lemma}

Second, a strong finite generating set for the limit $\cB^{\infty}$ of a deformable family gives rise a strong finite generating set for $\cB$ after a suitable localization of the ring $F_K$:
\begin{lemma} \cite[Lemma 3.2]{CLII} \label{lem:passagestrong} Let $K\subseteq \mathbb{C}$ be at most countable, and let $\cB$ be a deformable family over $F_K$ with weight grading $\cB = \bigoplus_{d\geq 0} \cB[d]$, such that $\cB[0] \cong F_K$. Let $U = \{\alpha_i|\ i\in I\}$ be a strong generating set for $\cB^{\infty}$, and let $T = \{a_i|\ i\in I\}$ be a subset of $\cB$ such that $\phi(a_i) = \alpha_i$. There exists a subset $S\subseteq \mathbb{C}$ containing $K$ which is at most countable, such that $F_S \otimes_{F_K}\cB$ is strongly generated by $T$. Here we have identified $T$ with the set $\{1 \otimes a_i|\ i\in I\} \subseteq F_S \otimes_{F_K} \cB$. \end{lemma}

Third, by \cite[Thm 4.10]{CLIII}, for any finite tensor product $\cV = \bigotimes_{i=1}^m \cV_i$ of standard free field algebras, and any reductive group $G$ of automorphisms of $\cV$, $\cV^G$ is strongly finitely generated.

We will also need a version of Lemma \ref{lem:passagestrong} for weak generating sets.
\begin{lemma}  \label{lem:passageweak} Let $\cB$ be a deformable family over $F_K$ as in Lemma \ref{lem:passagestrong}. Let $U = \{\alpha_i|\ i\in I\}$ be a weak generating set for $\cB^{\infty}$, and let $T = \{a_i|\ i\in I\}$ be a subset of $\cB$ such that $\phi(a_i) = \alpha_i$. There exists a subset $S\subseteq \mathbb{C}$ containing $K$ which is at most countable, such that $F_S \otimes_{F_K}\cB$ is weakly generated by $T$. As above, we have identified $T$ with the set $\{1 \otimes a_i|\ i\in I\} \subseteq F_S \otimes_{F_K} \cB$. \end{lemma}

\begin{proof} This is immediate from Lemma \ref{lem:passagestrong} and the fact that if $\{\alpha_i|\ i\in I\}$ weakly generates $\cB^{\infty}$, the set 
$$\{\alpha_{i_1} \circ_{j_1}( \cdots (\alpha_{i_r-1} \circ_{j_{r-1}} \alpha_{i_r})\cdots )|\ i_1,\dots, i_r \in I,\ j_1,\dots, j_{r-1} \geq 0\}$$ strongly generates $\cB^{\infty}$. \end{proof}

In this paper, we only need the special case of $\cW^k(\gs\gl_n)$, which satisfies
\begin{equation} \label{freefieldlimit} \cW^{\text{free}}(\gs\gl_n) \cong \bigotimes_{i=1}^n \cO(1, 2i).\end{equation}
We often use the notation $W^i$ for the generators of $\cO(1, 2i)$, which satisfy $$W^i(z) W^i(w) \sim (z-w)^{-2i}.$$

\section{Universal two-parameter $\cW_{\infty}$-algebra} \label{sec:winf} 
Here we recall some features of the universal two-parameter vertex algebra $\cW(c,\lambda)$ constructed by the second author in \cite{LII}. It is defined over the ring $\mathbb{C}[c,\lambda]$, and is generated by a Virasoro field $L$ of central charge $c$ and a primary weight $3$ field $W^3$ which is normalized so that $(W^3)_{(5)} W^3 = \frac{c}{3} 1$. The remaining strong generators $W^i$ of weight $i \geq 4$ are defined inductively by $$W^i = (W^{3})_{(1)} W^{i-1},\qquad i \geq 4.$$ Then $\cW(c,\lambda)$ is freely generated by $\{L,W^i|\ i\geq 3\}$. It has a conformal weight grading $$\cW(c,\lambda) = \bigoplus_{n\geq 0} \cW(c,\lambda)[n],$$ where each $\cW(c,\lambda)[n]$ is a free $\mathbb{C}[c,\lambda]$-module and $\cW(c,\lambda)[0] \cong \mathbb{C}[c,\lambda]$. There is a symmetric bilinear form on $\cW(c,\lambda)[n]$ given by
$$\langle ,  \rangle_n : \cW(c,\lambda)[n] \otimes_{\mathbb{C}[c,\lambda]} \cW(c,\lambda)[n] \rightarrow \mathbb{C}[c,\lambda],\qquad \langle \omega, \nu \rangle_n = \omega_{(2n-1)} \nu.$$ The determinant $\text{det}_n$ of this form is nonzero for all $n$, which is equivalent to the simplicity of $\cW(c,\lambda)$ as a vertex algebra over $\mathbb{C}[c,\lambda]$.

Certain coefficients are independent of the parameters $c,\lambda$. We have for $3 \leq i\leq j$
$$W^i(z) W^j(w) \sim \cdots + a_{i,j} W^{i+j-2}(w)(z-w)^{-2} + b_{i,j} \partial W^{i+j -2}(w)(z-w)^{-1},$$ where $a_{i,j}$ and $b_{i,j}$ are independent of $c,\lambda$. By definition, $a_{3,j} = 1$ for all $j\geq 3$.
In \cite{LII} we computed a few of these constants, namely, 
$$ a_{4,j} = \frac{4}{j+1},\qquad b_{4,j} =  \frac{12}{(j+1)(j+2)},\qquad a_{5,j} =  \frac{20}{(j+1)(j+2)}.$$ In fact, they can be deduced from the recursive procedure in the proof of Theorem 5.2 of \cite{LII}, and we record this result for later use.
\begin{lemma} \label{formofab} For all $3 \leq i \leq j$, we have
$$a_{i,j} = \frac{i!}{6 (j+1)(j+2)\cdots (j+i-3)}, \qquad b_{i,j} = \frac{i!(i-1)}{6 (j+1)(j+2)\cdots (j+i-2)}.$$
\end{lemma}

Let $p$ be an irreducible factor of $\text{det}_{N+1}$ and let $I = (p) \subseteq \mathbb{C}[c,\lambda] \cong \cW(c,\lambda)[0]$ be the corresponding ideal. Consider the quotient
$$ \cW^I(c,\lambda) = \cW(c,\lambda) / I \cdot \cW(c,\lambda),$$ where $I$ is regarded as a subset of the weight zero space $\cW(c,\lambda)[0] \cong \mathbb{C}[c,\lambda]$, and $I \cdot \cW(c,\lambda)$ denotes the vertex algebra ideal generated by $I$.  This is a vertex algebra over the ring $\mathbb{C}[c,\lambda]/I$, which is no longer simple. It contains a singular vector $\omega$ in weight $N+1$, which lies in the maximal proper ideal $\cI\subseteq \cW^I(c,\lambda)$ graded by conformal weight. If $p$ does not divide $\text{det}_{m}$ for any $m<N+1$, $\omega$ will have minimal weight among elements of $\cI$. Often, $\omega$ has the form \begin{equation} \label{sing:intro} W^{N+1} - P(L, W^3,\dots, W^{N-1}),\end{equation} possibly after localizing the ring $\mathbb{C}[c,\lambda]$, where $P$ is a normally ordered polynomial in the fields $L,W^3,\dots,$ $W^{N-1}$, and their derivatives. If this is the case, there will exist relations in the simple graded quotient $\cW_I(c,\lambda) :=\cW^I(c,\lambda) / \cI$ of the form
$$W^m = P_m(L, W^3, \dots, W^N),$$ for all $m \geq N+1$ expressing $W^m$ in terms of $L, W^3,\dots, W^N$ and their derivatives. Then $\cW_I(c,\lambda)$ will be of type $\cW(2,3,\dots, N)$. Conversely, any one-parameter vertex algebra $\cW$ of type $\cW(2,3,\dots, N)$ for some $N$ satisfying mild hypotheses, is isomorphic to $\cW_I(c,\lambda)$ for some $I = (p)$ as above, possibly after localizing. The corresponding variety $V(I) \subseteq\mathbb{C}^2$ is called the {\it truncation curve} for $\cW$.

\begin{thm} \cite[Thm 6.1]{CLIII} \label{c(n,m)asquotient} For $m \geq 1$ and $n\geq 0$, and for $m=0$ and $n\geq 3$, $\cC^{\psi}(n,m) \cong \cW_{I_{n,m}}(c,\lambda)$, where $I_{n,m}$ is described explicitly via the parametrization
\begin{equation} \label{cnmparametrization} \begin{split} c(\psi) &=  -\frac{(n \psi  - m - n -1) (n \psi - \psi - m - n +1 ) (n \psi +  \psi  -m - n)}{(\psi -1) \psi},
\\ \ 
\\ \lambda(\psi) & = -\frac{(\psi-1) \psi}{(n \psi - n - m -2) (n \psi - 2 \psi - m - n +2 ) (n \psi + 2 \psi  -m - n )}.
\end{split} \end{equation}
Moreover, after a suitable localization of the ring $\mathbb{C}[\psi]$, $\cW^{I_{n,m}}(c,\lambda)$ has a singular vector of the form
$$W^{(m+1)(m+n+1)} - P(L, W^3,\dots, W^{(m+1)(m+n+1)-1})$$ and no singular vector of lower weight, where $P$ is a normally ordered polynomial in the fields $L, W^3,\dots, W^{(m+1)(m+n+1)-1}$, and their derivatives. Therefore $\cW_{I_{n,m}}(c,\lambda)$ has minimal strong generating type $\cW(2,3,\dots, (m+1)(m+n+1)-1)$.
\end{thm}

It is expected that this list accounts for all the finite truncations of $\cW(c,\lambda)$, but this remains an open question.

\subsection{Orbifolds of $\cW(c,\lambda)$} By \cite[Cor. 5.3]{LII}, $\cW(c,\lambda)$ has full automorphism group $\mathbb{Z}_2$ as a vertex algebra over $\mathbb{C}[c,\lambda]$, and the action of the generator $\theta \in \mathbb{Z}_2$ is given by $\theta(L) = L$ and $\theta(W^3) = -W^3$. This forces  $\theta(W^i) = (-1)^i W^i$ for all $i \geq 3$. The action of $\mathbb{Z}_2$ on $\cW(c,\lambda)$ induces a $\mathbb{Z}_2$-grading 
$$\cW(c,\lambda) = \cW(c,\lambda)_{\overline{0}} \oplus  \cW(c,\lambda)_{\overline{1}},$$ where $\theta$ acts on $\cW(c,\lambda)_{\overline{0}}$ and $\cW(c,\lambda)_{\overline{1}}$ by $\text{id}$ and $-\text{id}$, respectively. Therefore $\cW(c,\lambda)^{\mathbb{Z}_2} = \cW(c,\lambda)_{\overline{0}}$. Moreover, all the ideals $\cI_{n,m}$ are graded by $\mathbb{Z}_2$:
$$\cI_{n,m} = \cI_{n,m,\overline{0}} \oplus \cI_{n,m,\overline{1}},\ \text{where} \  \cI_{n,m,\overline{0}} = \cI_{n,m} \cap \cW(c,\lambda)_{\overline{0}}, \ \text{and} \  \cI_{n,m,\overline{1}} = \cI_{n,m} \cap \cW(c,\lambda)_{\overline{1}}.$$
It follows that $\cC^{\psi}(n,m)^{\mathbb{Z}_2}$ can be realized as the quotient 
\begin{equation} \label{quotientrealization} \cC^{\psi}(n,m)^{\mathbb{Z}_2} \cong \cW(c,\lambda)^{\mathbb{Z}_2} /  \cI_{n,m,\overline{0}}.\end{equation}
If we pass to a suitable localization $R$ of $\mathbb{C}[c,\lambda]$, the weight $d$ subspace $\cW(c,\lambda)^{\mathbb{Z}_2}[d]$ will be a free $R$-module, and a basis $\{\omega_1,\dots, \omega_{r_d}\}$ for this space will descend to a basis for the weight $d$ subspace $\cC^{\psi}(n,m)^{\mathbb{Z}_2}[d]$, for all but finitely many choices of $n,m$.

Define the following fields
\begin{equation} \label{def:uijab} U^{2i+1, 2j+1}_{a,b} =\ :(\partial^a W^{2i+1})( \partial^b W^{2j+1}):,\qquad 1\leq i \leq j, \qquad a,b \geq 0,\end{equation} which have weight $2i+2j+a+b+2$ and clearly lie in $\cW(c,\lambda)^{\mathbb{Z}_2}$. Note that for $1 \leq i < j$ and $m\geq 0$, the following sets span the same $(m+1)$-dimensional space of weight $2i+2j+m + 2$:
\begin{equation} \label{changeofvariable1} 
\{ U^{2i+1,2j+1}_{a, m-a}|\ 0\leq m \leq a\},\quad \{ \partial^a U^{2i+1,2j+1}_{0,m-a}|\ 0\leq a \leq m\},\quad \{ \partial^a U^{2i+1,2j+1}_{m-a,0}|\ 0\leq a \leq m\}. \end{equation} Similarly, for $i\geq 1$ and $m\geq 0$, the following sets span the same $(m+1)$-dimensional spaces of weights $4i+2m+2$ and $4i+2m+3$, respectively:
\begin{equation} \label{changeofvariable2} \begin{split} & \{U^{2i+1,2i+1}_{a, 2m-a}|\ 0\leq a \leq 2m\},\qquad \{\partial^{2a} U^{2i+1,2i+1}_{0,2m-2a}|\ 0\leq a \leq m\},
\\ &  \{U^{2i+1,2i+1}_{a, 2m+1-a}|\ 0\leq m \leq a\}, \qquad \{ \partial^{2a+1} U^{2i+1,2i+1}_{0,2m-2a}|\ 0\leq a \leq m\}.\end{split}
 \end{equation}

\begin{lemma} \label{lem:longlistwinf} As a vertex algebra over the ring $\mathbb{C}[c,\lambda]$, $\cW(c,\lambda)^{\mathbb{Z}_2}$ has a strong generating set consisting of the union of the following sets:
\begin{enumerate}
\item $\{L, W^{2i}|\ i \geq 2\}$,
\item $\{U^{2i+1,2i+1}_{0,2a} |\ i\geq 1,\ a \geq 0\}$,
\item $\{U^{2i+1,2j+1}_{0,a} |\ 1 \leq i < j, \ a \geq 0\}$.
\end{enumerate}
In particular, this holds for all $c,\lambda$, so no localization of $\mathbb{C}[c,\lambda]$ is required.
\end{lemma}

\begin{proof} The key observation is that $\cW(c,\lambda)$ has a good increasing filtration in the sense of \cite{Li}, defined by $\text{deg}(W^i) = i$ for $i \geq 2$, where $W^2 = L$, and setting the degree of a monomial $:\partial^{k_1}W^{i_1} \cdots \partial^{k_r} W^{i_r}:$ to be at most $i_1 + \cdots + i_r$. It is apparent from the OPE algebra which is defined recursively in \cite{LII} that for all $i,j \geq 2$ and $k\geq 0$, the degree of $W^i_{(k)}W^j$ is at most $i+j -2$. It follows that the associated graded algebra $\text{gr}(\cW(c,\lambda))$ is the differential commutative algebra with generators $\{W^i|\ i \geq 2\}$. This filtration is clearly $\mathbb{Z}_2$-invariant, hence $\text{gr}(\cW(c,\lambda)^{\mathbb{Z}_2}) \cong \text{gr}(\cW(c,\lambda))^{\mathbb{Z}_2}$. Since the latter is clearly generated by the above monomials (where normally ordered is now unnecessary), the claim follows.
\end{proof}

Since $\cC^{\psi}(n,m)$ is of type $\cW(2,3,\dots, (m+1)(m+n+1)-1)$ after localizing the ring $\mathbb{C}[\psi]$, we obtain
\begin{cor} \label{cor:longlistcnm}  For $n\geq 3$ and $m=0$, and for $n,m\geq 1$, after the above localization, $\cC^{\psi}(n,m)^{\mathbb{Z}_2}$ has a strong generating set
\begin{enumerate}
\item $L, W^{2i}$, for $ 2 \leq i \leq  \frac{(m+1)(m+n+1)-1}{2}$,
\item $U^{2i+1,2i+1}_{0,2a}$, for all $1 \leq i <  \frac{(m+1)(m+n+1)-1}{2}$ and $a \geq 0$,
\item $U^{2i+1,2j+1}_{0,a}$, for all $1 \leq i < j <  \frac{(m+1)(m+n+1)-1}{2}$ and $a \geq 0$.
\end{enumerate}
\end{cor}

Specializing to the case $n\geq 3$ and $m=0$, we have 
\begin{cor} \label{cor:longlistwsln} After the above localization, $\cW^k(\gs\gl_n)^{\mathbb{Z}_2}$ has a strong generating set consisting of the union of the following sets:
\begin{enumerate}
\item $\{L, W^{2i}|\ 2 \leq i \leq \frac{n}{2}\}$,
\item $\{U^{2i+1,2i+1}_{0,2a}|\ 1 \leq i < \frac{n}{2} ,\ a \geq 0\}$,
\item $\{U^{2i+1,2j+1}_{0,a} |\ 1 \leq i < j < \frac{n}{2}, \ a \geq 0\}$.
\end{enumerate}
\end{cor}

\section{The structure of $\cW^k(\gs\gl_n)^{\mathbb{Z}_2}$}
Before we consider the structure of $\cW(c,\lambda)^{\mathbb{Z}_2}$ and its quotients $\cC^{\psi}(n,m)^{\mathbb{Z}_2}$ in general, we begin by studying the special case $\cC^{\psi}(n,0)^{\mathbb{Z}_2} \cong \cW^k(\gs\gl_n)^{\mathbb{Z}_2}$. This has the advantage that its large level limit is an orbifold of a free field algebra. In the case $n=3$, $\cW^{k}(\gs\gl_3)^{\mathbb{Z}_2}$ has a (minimal) strong generating set $\{L, U^{3,3}_{0,2a}|\ a = 0,1,2,3\}$, and hence is of type $\cW(2,6,8,10,12)$ \cite{AL}. For the rest of this section, we assume that $n\geq 4$. 

Recall that $\cW^k(\gs\gl_n)^{\mathbb{Z}_2}$ is a deformable family, and 
$$\cW^{\text{free}}(\gs\gl_n)^{\mathbb{Z}_2} \cong \lim_{k\rightarrow \infty} \cW^k(\gs\gl_n)^{\mathbb{Z}_2} \cong \bigg(\lim_{k\rightarrow \infty} \cW^k(\gs\gl_n)\bigg)^{\mathbb{Z}_2} \cong  \bigg( \bigotimes_{i=1}^n \cO(1, 2i)\bigg)^{\mathbb{Z}_2}.$$ Moveover, since $\theta \in \mathbb{Z}_2$ acts on the generator $W^i \in \cO(1,2i)$ by $\theta(W^i) = (-1)^i W^i$, we may rewrite this in the form
$$\cW^{\text{free}}(\gs\gl_n)^{\mathbb{Z}_2} \cong 
  \left\{
\begin{array}{ll}
\bigotimes_{i=1}^{n/2} \cO(1,4i) \bigotimes \bigg(\bigotimes_{i=1}^{n/2} \cO(1,4i-2)\bigg)^{\mathbb{Z}_2},\  & n\ \text{even},
\bigskip
\\ 
\bigotimes_{i=1}^{(n-1)/2} \cO(1,4i) \bigotimes \bigg(\bigotimes_{i=1}^{(n-1)/2}\cO(1,4i+2)\bigg)^{\mathbb{Z}_2},\  & n\ \text{odd}.
\\
\end{array} 
\right.$$ 
We shall use the same notation $U^{2i+1, 2j+1}_{a,b}$ to denote the elements $:(\partial^a W^{2i+1})( \partial^b W^{2j+1}):$ in both $\cW^k(\gs\gl_n)^{\mathbb{Z}_2}$ and $\cW^{\text{free}}(\gs\gl_n)^{\mathbb{Z}_2}$ when no confusion can arise.

\subsection{Weak generators for $\cW^k(\gs\gl_n)^{\mathbb{Z}_2}$} First, we will find a minimal weak generating set for $\cW^{\text{free}}(\gs\gl_n)^{\mathbb{Z}_2}$.
\begin{thm} \label{thm:longlistweakwslnfree} For all $n\geq 4$, $\cW^{\text{free}}(\gs\gl_n)^{\mathbb{Z}_2}$ has a minimal weak generating set $S$ consisting of the union of the following sets:
\begin{enumerate}
\item $\{W^{2i}|\ 1 \leq i \leq \frac{n}{2}\}$,
\item $\{U^{3,2i+1}_{0,0}|\ 1 \leq i < \frac{n}{2}\}$.
\end{enumerate}
\end{thm}

\begin{proof}
First, we need to show that $S$ is indeed a weak generating set. It is straightforward to check that for all $a\geq 0$,
$$(U^{3,3}_{0,0})_{(3)} U^{3,3}_{0,2a} = \frac{(4 + a) (15 + 8 a + 4 a^2) }{30} U^{3,3}_{0,2a+2} + \cdots,$$ where remaining terms are of the form $\partial^{2i} U^{3,3}_{0,2a+2 -2i}$ for $i = 1,2,\dots,a+1$. It follows by induction that all terms $U^{3,3}_{0,2a}$ can be generated from $U^{3,3}_{0,0}$.

Next, for all $a \geq 0$ and $2\leq i < \frac{n}{2}$, we have 
$$(U^{3,3}_{0,0})_{(4)} U^{3,2i+1}_{a,0} = \frac{(2 + a) (3 + a) (4 + a) (5 + a)}{60} U^{3,2i+1}_{a+1,0} + \cdots ,$$ where remaining terms are of the form $\partial^{i} U^{3,2i+1}_{0,a+1-i}$ for $i = 1,2,\dots,a+1$. Since the sets $\{\partial^{i} U^{3,2i+1}_{0,a-i}|\ i = 0,1,\dots, a\}$ and $\{\partial^{i} U^{3,2i+1}_{a-i,0}|\ i = 0,1,\dots, a\}$ span the same vector space, it follows that all fields $U^{3,2i+1}_{0,a}$ can be generated from $S$.

Next, for all $2\leq i < \frac{n}{2}$, we have 
$$(U^{3,2i+1}_{0,0})_{(5)} U^{3,2i+1}_{0,0} = U^{2i+1,2i+1}_{0,0} + \frac{1}{ (4 i -5)! 4 ( i -1)} U^{3,3}_{4i-4,0}.$$
This shows that $U^{2i+1,2i+1}_{0,0}$ lies in the algebra generated by $S$. By the same argument as above, $U^{2i+1,2i+1}_{0,0}$ generates $U^{2i+1,2i+1}_{0,2a}$ for all $a \geq 1$.

Finally, for all $2\leq i < j < \frac{n}{2}$ and $a\geq 0$, we have 
$$(U^{3,2i+1}_{0,0})_{(5)} U^{3,2j+1}_{a,0} = \frac{(1 + a) (2 + a) (3 + a) (4 + a) (5 + a) }{120} U^{2i+1, 2j+1}_{a,0}.$$ Since all the strong generators appearing in Corollary \ref{cor:longlistwsln} can be generated by $S$, this shows that $S$ is indeed a weak generating set. 

As for minimality, all the generators $W^{2i}$ are needed because they do not appear in the OPEs of the other generators $W^j$ for $\cW^{\text{free}}(\gs\gl_n)$, for $j \neq 2i$. Moreover, if any of the generators $U^{3,2i+1}_{0,0}$ were omitted, we would not be able to generate any field $U^{2i+1, 2i+i}_{a,b}$, $U^{2i+1, 2j+i}_{a,b}$ for $i<j$, or $U^{2j+1, 2i+i}_{a,b}$ for $j<i$, since $W^{2i+1}$ does not appear in the OPEs of $W^{\ell}$ for $\ell \neq 2i+1$. \end{proof}

\begin{remark} We may also consider the infinite tensor product $\bigotimes_{i=1}^{\infty} \cO(1, 2i)$, which has a $\mathbb{Z}_2$-action given by $\theta(W^i) = (-1)^i W^i$ for all $i\geq 2$. Then the same argument shows that $\big(\bigotimes_{i=1}^{\infty} \cO(1, 2i) \big)^{\mathbb{Z}_2}$ has a minimal weak generating set $\{W^{2i}, U^{3,2i+1}_{0,0}|\ i\geq 1\}$.
\end{remark}

An immediate consequence of Lemma \ref{lem:passageweak} and Theorem \ref{thm:longlistweakwslnfree} is 
\begin{cor} \label{cor:longlistweakwslnfree} After a suitable localization, for all $n\geq 4$, $\cW^{k}(\gs\gl_n)^{\mathbb{Z}_2}$ has a weak generating set consisting of the union of the following sets:
\begin{enumerate}
\item $\{W^{2i}|\ 1 \leq i \leq \frac{n}{2}\}$,
\item $\{U^{3,2i+1}_{0,0}|\ 1 \leq i < \frac{n}{2}\}$.
\end{enumerate}
\end{cor}
We will see later that this weak generating set is not minimal, and we will reduce it to a minimal weak generating set.

\subsection{Strong generators for $\cW^{k}(\gs\gl_n)^{\mathbb{Z}_2}$} Next, we will find a strong generating set for $\cW^{k}(\gs\gl_n)^{\mathbb{Z}_2}$, again starting with the large level limit.
\begin{thm}  \label{thm:strongwnfree} 
For $n \geq 4$, $\cW^{\text{free}}(\gs\gl_n)^{\mathbb{Z}_2}$ has a minimal strong generating set $T$ consisting of
\begin{enumerate}
\item $W^{2i}$, for $1\leq i \leq \frac{n}{2}$,
\item $U^{3,3}_{0,2a}$, for $a = 0,1,2,3$,
\item $U^{3,2i+1}_{0,a}$, for all $2 \leq i < j < \frac{n}{2}$ and $a = 0,1,2,3,4,5,6$,
\item $U^{2i+1,2i+1}_{0,2a}$, for all $2\leq i < \frac{n}{2}$ and $a = 0,1,2$,
\item $U^{2i+1,2j+1}_{0,a}$, for all $2 \leq i < j < \frac{n}{2}$ and $a = 0,1,2,3,4,5$.
\end{enumerate}
\end{thm}
 
 \begin{proof} First, we need to show that $T$ is indeed a strong generating set. In order to handle the cases of $n$ even and odd simultaneously, we write $d = \frac{n}{2}-1$ when $n$ is even and $d  = \frac{n-1}{2}$ when $n$ is odd. In both cases we need to find minimal strong generators of $\big(\bigotimes_{i=1}^{d} \cO(1,4i+2)\big)^{\mathbb{Z}_2}$. In view of Corollary \ref{cor:longlistwsln}, it suffices to construct decoupling relations for 
 \begin{enumerate}
 \item $U^{3,3}_{2a,0}$, for $a\geq 4$, 
 \item $U^{3,2i+1}_{0,a}$, for $2 \leq i < j \leq d$ and $a\geq 7$,
 \item $U^{2i+1,2i+1}_{0,2a}$, for $2\leq i \leq d$ and $a\geq 3$,
 \item $U^{2i+1,2j+1}_{0,a}$, for $2 \leq i < j \leq d$ and $a \geq 6$.
\end{enumerate}
In other words, these fields can all be expressed as normally ordered polynomials in the elements of $T$ and their derivatives. 

A calculation shows that we have the following relation in weight $14$:
 \begin{equation}\label{eq:n=4wt14}
:{U}^{3,3}_{0,0}{U}^{3,3}_{1,1}:-:{U}^{3,3}_{0,1}{U}^{3,3}_{0,1}:\ =-\frac{19}{4032} U^{3,3}_{0,8}+ \frac{23}{1440} \partial^2 U^{3,3}_{0,6}  -\frac{23}{576} \partial^4 U^{3,3}_{0,4} + \frac{23}{480}  \partial^6 U^{3,3}_{0,2} -\frac{391}{40320} \partial^8 U^{3,3}_{0,0}.\end{equation}
Since $U^{3,3}_{0,1}=\frac{1}{2} \partial U^{3,3}_{0,0}$ and $U^{3,3}_{1,1}=-U^{3,3}_{0,2}+\frac{1}{2}\partial ^{2}U^{3,3}_{0,0}$,
\eqref{eq:n=4wt14} can be written in the form
\begin{equation}\label{eq6}
{U}^{3,3}_{0,8}=P_{4}({U}^{3,3}_{0,0},{U}^{3,3}_{0,2},{U}^{3,3}_{0,4},{U}^{3,3}_{0,6}),
\end{equation} where $P_4$ is a normally ordered polynomial in $\{U_{2a,0}|\ a = 0,1,2,3\}$ and their derivatives. This is the desired decoupling relation for ${U}^{3,3}_{0,8}$.

Similarly, we have the following relation in weight $16$:
\begin{equation}\label{eq:n=4wt16} \begin{split} 
:{U}^{3,3}_{0,0}{U}^{3,3}_{2,2}:-:{U}^{3,3}_{0,2}{U}^{3,3}_{0,2}: \ = & -\frac{1}{7200} U^{3,3}_{0,10} -\frac{1}{72} \partial^2 U^{3,3}_{0,8} + \frac{7}{96} \partial^4 U^{3,3}_{0,6}   -\frac{7}{32} \partial^6 U^{3,3}_{0,4} \\ & + \frac{17}{64} \partial^8 U^{3,3}_{0,2}   -\frac{31}{576} \partial^{10} U^{3,3}_{0,0}. \end{split}
\end{equation}
Since $U^{3,3}_{2,2}=U^{3,3}_{0,4}-2\partial ^{2}U^{3,3}_{0,2}+\frac{1}{2} \partial U^{3,3}_{0,0}$, we can use this and \eqref{eq6}, to rewrite \eqref{eq:n=4wt16} in the form
 \begin{equation}\label{eq7*}
{U}^{3,3}_{0,10}=P_{5}({U}^{3,3}_{0,0},{U}^{3,3}_{0,2},{U}^{3,3}_{0,4},{U}^{3,3}_{0,6}).
\end{equation}
This is the desired decoupling relation for ${U}^{3,3}_{0,10}$.

We now assume inductively that there exist similar decoupling relation 
\begin{equation} \label{eq:n=4induction} {U}^{3,3}_{0,2a}=P_{a}({U}^{3,3}_{0,0},{U}^{3,3}_{0,2},{U}^{3,3}_{0,4},{U}^{3,3}_{0,6}),\end{equation} for all $0\leq a \leq n$, with the cases $a = 0,1$ being the base cases. Observe that the operator $(U^{3,3}_{0,0})_{(1)}$ raises weight by $4$, and satisfies
\begin{equation}\label{n=4raising}
(U^{3,3}_{0,0})_{(1)} U^{3,3}_{0,2m}= \frac{5 + m}{30}U^{3,3}_{0,2m+4}+ \cdots,
\end{equation} where the remaining term is a linear combination of $\partial^{2i} U^{3,3}_{0,2m+4-2i}$ for $1\leq i \leq m+2$.
Applying it to the relation ${U}^{3,3}_{0,2n-2}=P_{n-1}({U}^{3,3}_{0,0},{U}^{3,3}_{0,2},{U}^{3,3}_{0,4},{U}^{3,3}_{0,6})$ yields
\begin{equation}\label{n=4raising2} \frac{3+ 2n}{30}U^{3,3}_{0,2n+2}+ \cdots = (U^{3,3}_{0,0})_{(1)} P_{n-1}({U}^{3,3}_{0,0},{U}^{3,3}_{0,2},{U}^{3,3}_{0,4},{U}^{3,3}_{0,6}).\end{equation} Moreover, since $ (U^{3,3}_{0,2a})_{(0)}  U^{3,3}_{0,2b}$ is a total derivative for all $a,b,r,s$, it is apparent that the right hand side of \eqref{n=4raising2} does not depend on $U^{3,3}_{0,2n+2}$. Therefore we can rescale \eqref{n=4raising2} and rewrite it in the form
$$U^{3,3}_{0,2n+2} = Q(U^{3,3}_{0,0}, U^{3,3}_{0,2},\dots, U^{3,3}_{0,2n}),$$ where $Q$ is a normally ordered polynomial in $U^{3,3}_{0,0}, U^{3,3}_{0,2},\dots, U^{3,3}_{0,2n}$ and their derivatives. Finally, using the relations \eqref{eq:n=4induction} for $a \leq n$, we can rewrite this in the form 
$${U}^{3,3}_{0,2a}=P_{a}({U}^{3,3}_{0,0},{U}^{3,3}_{0,2},{U}^{3,3}_{0,4},{U}^{3,3}_{0,6}),$$ as desired. Therefore we have decoupling relations for $U^{3,3}_{0,2a}$ for $a\geq 4$.

Next, we compute
 \begin{equation}\label{odd1}
 \begin{split}
& :{U}^{3,3}_{0,0}{U}^{3,2i+1}_{1,0}:\ -\ :{U}^{3,3}_{0,1}{U}^{3,2i+1}_{0,0}:\ =
\frac{11}{5040}{U}^{3,2i+1}_{0,7},
\\ & :{U}^{3,3}_{0,0}{U}^{3,2i+1}_{2,0}:\ \ -:{U}^{3,3}_{0,2}{U}^{3,2i+1}_{0,0}:\ =
\frac{1}{2880}{U}^{3,2i+1}_{0,8},\end{split}
\end{equation}
which are decoupling relations for $U^{3,2i+1}_{0,7}$ and $U^{3,2i+1}_{0,8}$, for all $i\geq 2$. 

We shall now construct higher decoupling relations 
\begin{equation} \label{eq:decoup3i} U^{3,2i+1}_{0,r} = P_{i,r}, \ \text{for all}\ i\geq 2\ \text{and} \ r > 8,\end{equation} where $P_{i,r}$ is a normally ordered polynomial in the elements of $T$ and their derivatives. We regard $\cO_{\text{ev}}(1, 4i+2)$ as a subalgebra of the rank one Heisenberg algebra $\cH(1)$ with generator $\alpha^{2i+1}$ satisfying
$$ \alpha^{2i+1} (z) \alpha^{2i+1}(w) \sim (z-w)^{-2}.$$ Then $\bigotimes_{i=1}^{d} \cO_{\text{ev}}(1, 4i+2)$ is a subalgebra of the rank $d$ Heisenberg algebra $\cH(d)$, where
the generators $W^{2i+1}$ are given by
$$W^{2i+1} = \frac{\epsilon}{\sqrt{(4i+1)!}} \partial^{2i} \alpha^{2i+1}.$$
Next, let $$\nu = \sum_{i=1}^d :(\partial^2 \alpha^{2i+1}) \alpha^{2i+1}: \ \in  \cH(d)^{\mathbb{Z}_2}.$$ 
Note that $\nu$ does not lie in the subalgebra $\big(\bigotimes_{i=1}^{d} \cO_{\text{ev}}(1, 4i+2)\big)^{\mathbb{Z}_2}$; however, the mode $\nu_{(1)}$ preserves this subalgebra and raises the weight by $2$. A calculation shows that for all $a \geq 0$,
\begin{equation} \label{outsideoperator} \nu_{(1)} U^{3, 2i+1}_{0,a} = (12 + 2 a + 4 i)U^{3, 2i+1}_{0,a+2} + \cdots,\end{equation} where the remaining terms are of the form $\partial^{b}U^{3, 2i+1}_{0,a+2-b}$ for $1\leq b \leq a+2$. 

By applying $\nu_{(1)}$ repeatedly to \eqref{odd1}, we can inductively construct all decoupling relations
$$U^{3,2i+1}_{0,r} = P_{i,r},$$ using the same procedure as the construction of the relations for $U^{3,3}_{0,2a}$ above.

Next, we have the following calculation:
\begin{equation} \begin{split} 
& :{U}^{2i+1,2i+1}_{0,r}{U}^{2j+1,2j+1}_{0,0}:-:{U}^{2i+1,2j+1}_{0,r}{U}^{2i+1,2j+1}_{0,0}:\ =
\\ &-\frac{1}{2(2j+1)!}{U}^{2i+1,2i+1}_{0,4j+2+r} -\frac{1}{(4i+1)!(4i+2 +r)}{U}^{2j+1, 2j+1}_{0,4i+2+r}.
 \end{split}\end{equation}
 Specializing to the case $i = 1$ and $r = 2a$, we get the relation
 \begin{equation} \begin{split} &
:{U}^{3,3}_{0,2a}{U}^{2j+1,2j+1}_{0,0}:-:{U}^{3,2j+1}_{0,r}{U}^{3,2j+1}_{0,0}:\ =-\frac{1}{2(2j+1)!}{U}^{3,3}_{0,4j+2+2a} -\frac{1}{5!(6 +2a)}{U}^{2j+1, 2j+1}_{0,6+2a}.
 \end{split}\end{equation}
Since we already have decoupling relations for ${U}^{3,3}_{0,4j+2+2a}$, this yields the desired decoupling relations for ${U}^{2j+1, 2j+1}_{0,6+2a}$, for all $j\geq 2$ and $a\geq 0$.
 
Finally, for $2\leq i<j\leq d$ we compute
\begin{equation}\label{lastone}
:{U}^{3,2i+1}_{0,r}{U}^{3,2j+1}_{0,0}:\ -\ :{U}^{3,3}_{0,0}{U}^{2i+1,2j+1}_{4,0}:\ =
\frac{1}{720}{U}^{2i+1,2j+1}_{6+r,0}.
 \end{equation}
 Since we already have decoupling relations for ${U}^{3,2i+1}_{0,r}$ for all $r\geq 7$, this yields decoupling relations for ${U}^{2i+1, 2j+1}_{6+r,0}$, for all $2\leq i < j \leq d$ and $r \geq 0$. The same then holds for ${U}^{2i+1, 2j+1}_{0,6+r}$. This completes the proof that $T$ is a strong generating set for $\cW^{\text{free}}(\gs\gl_n)^{\mathbb{Z}_2}$.
 
As for the minimality, there can be no decoupling relation for $U^{3,3}_{0,2a}$ for $a=0,1,2$ since there are no relations of weight less than $14$. Similarly, if there were a decoupling relation for $U^{3,2i+1}_{0,a}$ for some $a \leq 6$, it would need to involve the variable $W^{2i+1}$ and have weight at most $2i+4 + a$. The relation of this kind of minimal weight has the form $:U^{3,3}_{0,0} U^{3,2i+1}_{1,0} :- :U^{3,3}_{1,0} U^{3,2i+1}_{0,0}: \ = \cdots$, but this relation has weight $2i+11$. The other possible decoupling relations are ruled out similarly.
 \end{proof}

\begin{remark} We may also consider the infinite tensor product $\bigotimes_{i=1}^{\infty} \cO(1, 2i)$, which has a $\mathbb{Z}_2$-action given by $\theta(W^i) = (-1)^i W^i$ for all $i\geq 2$. The same argument shows that $\big(\bigotimes_{i=1}^{\infty} \cO(1, 2i) \big)^{\mathbb{Z}_2}$ has a minimal strong generating set consisting of
\begin{enumerate}
\item $W^{2i}$, for $i \geq 1$,
\item $U^{3,3}_{0,2a}$, for $a = 0,1,2,3$,
\item $U^{3,2i+1}_{0,a}$, for all $i\geq 2$ and $a = 0,1,2,3,4,5,6$,
\item $U^{2i+1,2i+1}_{0,2a}$, for all $i \geq 2$ and $a = 0,1,2$,
\item $U^{2i+1,2j+1}_{0,a}$, for all $2 \leq i < j$ and $a = 0,1,2,3,4,5$.
\end{enumerate}
In particular, $\big(\bigotimes_{i=1}^{\infty} \cO(1, 2i) \big)^{\mathbb{Z}_2}$ is of type $\cW(2, 4, 6^2, 8^3, 9, 10^5, \dots k^{n_k},\dots)$. In this notation, for all integers $k\geq 11$,
$$n_k =   \left\{
\begin{array}{lll}
3m-2 &  & k = 4m,
\smallskip
\\ 3m-5  &  & k = 4m+1,
\smallskip
\\ 3m &  & k = 4m+2,
\smallskip
\\ 3m - 4 &  & k = 4m+3.
\end{array} 
\right.$$
\end{remark}

The following is immediate from Lemma \ref{lem:passagestrong} and Theorem \ref{thm:strongwnfree}.
\begin{cor}  \label{cor:strongwn} 
After a suitable localization, for all $n\geq 4$, $\cW^{\text{k}}(\gs\gl_n)^{\mathbb{Z}_2}$ has a strong generating set consisting of
\begin{enumerate}
\item $W^{2i}$, for $1\leq i \leq \frac{n}{2}$,
\item $U^{3,3}_{0,2a}$, for $a = 0,1,2,3$,
\item $U^{3,2i+1}_{0,a}$, for all $2 \leq i < j < \frac{n}{2}$ and $a = 0,1,2,3,4,5,6$,
\item $U^{2i+1,2i+1}_{0,2a}$, for all $2\leq i < \frac{n}{2}$ and $a = 0,1,2$,
\item $U^{2i+1,2j+1}_{0,a}$, for all $2 \leq i < j < \frac{n}{2}$ and $a = 0,1,2,3,4,5$.
\end{enumerate}
\end{cor}

One might speculate that the strong generating set for $\cW^k(\gs\gl_n)^{\mathbb{Z}_2}$ given by Corollary \ref{cor:strongwn} is minimal for generic values of $k$ as well. In fact, this is the case for $n = 4,5,6$, which was first conjectured in the physics literature in \cite{B-H}.
 \begin{cor} \label{cor:phys456} After a suitable localization, for $n=4,5,6$, $\cW^k(\gs\gl_n)^{\mathbb{Z}_2}$ has the following minimal strong generating type:
 \begin{enumerate}
\item $\cW^k(\mathfrak{sl}_4)^{\mathbb{Z}_2}$ is of type $\cW(2,4,6,8,10,12)$.
\item $\cW^k(\mathfrak{sl}_5)^{\mathbb{Z}_2}$ is of type $\cW(2,4,6,8^2,9,10^3,11,12^3,13,14^2)$.
\item $\cW^k(\mathfrak{sl}_6)^{\mathbb{Z}_2}$ is of type $\cW(2,4,6^2,8^2,9,10^3,11,12^3,13,14^2)$.
 \end{enumerate}
 \end{cor}

 \begin{proof} In the case $n=4$ there is nothing to prove because there are no normally ordered relations below weight $14$. For $n=5,6$, there are three fields of weight $14$  given by Corollary \ref{cor:longlistwsln}, namely $U^{3,3}_{0,8}$, $U^{5,5}_{0,4}$, $U^{3,5}_{0,6}$, and only one normally ordered relation in weight $14$, which has the form $:U^{3,3}_{0,0} U^{3,3}_{1,1}: - :U^{3,3}_{1,0} U^{3,3}_{1,0}: = \cdots$. As we have already seen, this can be used to eliminate one of these generators, namely $U^{3,3}_{0,8}$, and the other two are needed because there are no more relations in weight $14$. \end{proof} 

\begin{remark} In the case $n=4$, this statement follows alternatively from \cite[Cor. 6.1]{KL}, due to the isomorphism of Lie algebras $\gs\gl_4 \cong \gs\go_6$. In \cite[Section 10]{KL}, it was stated erroneously that the only nontrivial one-parameter quotient of $\cW(c,\lambda)^{\mathbb{Z}_2}$ which is isomorphic to a quotient of the even spin algebra $\cW^{\text{ev}}(c,\lambda)$, is $N^k(\gs\gl_2)$, which is the case $\cC^{\psi}(1,1)^{\mathbb{Z}_2}$ and is of type $\cW(2,4,6,8,10)$. The correct statement is that there are two such examples, namely, $\cC^{\psi}(1,1)^{\mathbb{Z}_2}$ and  $\cC^{\psi}(4,0)^{\mathbb{Z}_2}$.
\end{remark}

When $n \geq 7$, the strong generating set for $\cW^k(\gs\gl_n)^{\mathbb{Z}_2}$ given by Corollary \ref{cor:strongwn} is {\it not} minimal. To illustrate this phenomenon we consider the case of $n=7$ in detail.
\begin{thm} \label{thm:sl7min} After a suitable localization, $\cW^k(\gs\gl_7)^{\mathbb{Z}_2}$ has a minimal strong generating set
\begin{enumerate}
\item $L, W^4, W^6$,
\item $U^{3,3}_{0,2a}$, for $a = 0,1,2,3$,
\item $U^{3,5}_{0,a}$ for $a = 0,1,2,3,4,5,6$,
\item $U^{3,7}_{0,a}$, for $a = 0,1,2,3,4,5$,
\item $U^{5,5}_{0,2a}$ for $a = 0,1,2$,
\item $U^{5,7}_{0,a}$, for $a = 0,1,2,3$,
\item $U^{7,7}_{0,2a}$, for $a = 0,1$,
\end{enumerate}
and is therefore of type $\cW(2,4,6^2,8^2, 9, 10^4, 11^2, 12^5, 13^3, 14^5, 15^2, 16)$. In particular, the remaining fields in weights $16$, $17$, and $18$ which appear in Corollary \ref{cor:strongwn}, are not needed for $\cW^k(\gs\gl_7)^{\mathbb{Z}_2}$.
\end{thm}

\begin{proof}
The same argument as the proof of Corollary \ref{cor:phys456} shows that there are no further decoupling relations in weight $14$ (or below). Therefore the five fields given by Corollary \ref{cor:strongwn}, namely, $U^{3,5}_{0,6}$, $U^{5,5}_{0,4}$, $U^{3,7}_{0,4}$, $U^{5,7}_{0,2}$, and $U^{7,7}_{0,0}$, are all needed.

In weight $15$, there are two fields given by Corollary \ref{cor:strongwn}, $U^{3,7}_{0,5}$ and $U^{5,7}_{0,3}$. A priori, there is only one possible relation that might allow further decoupling, namely,
$$:U^{3,3}_{0,0} U^{3,3}_{2,1}: - :U^{3,3}_{2,0} U^{3,3}_{1,0}: \ = \cdots.$$ However, a computation shows that this relation does not allow either $U^{3,7}_{0,5}$ and $U^{5,7}_{0,3}$ to be decoupled, so both are needed in $\cW^k(\gs\gl_7)^{\mathbb{Z}_2}$.

In weight $16$ there are six generators from Corollary \ref{cor:longlistwsln}, namely $U^{3,3}_{0,10}$, $U^{3,5}_{0,8}$, $U^{3,7}_{0,6}$, $U^{5,5}_{0,6}$, $U^{5,7}_{0,4}$, and $U^{7,7}_{0,2}$. Three of these generators, namely $U^{3,3}_{0,10}$, $U^{3,5}_{0,8}$, and $U^{5,5}_{0,6}$, can be eliminated in $\cW^{\text{free}}(\gs\gl_7)^{\mathbb{Z}_2}$, and therefore in $\cW^k(\gs\gl_7)^{\mathbb{Z}_2}$ for generic $k$, using the relations
\begin{equation} \begin{split}  :U^{3,5}_{0,0} U^{3,5}_{0,0}: -: U^{3,3}_{0,0} U^{5,5}_{0,0} &  = \cdots, 
\\  :U^{3,3}_{0,0} U^{3,5}_{2,0}: - :U^{3,3}_{2,0} U^{3,5}_{0,0}: &  = \cdots, 
\\  :U^{3,3}_{0,0} U^{3,3}_{2,2}: - :U^{3,3}_{2,0} U^{3,3}_{2,0}: &  = \cdots.
\end{split}\end{equation}
Note that there is no relation in weight $16$ that involves the field $W^7$. For any relation of weight $16$ among the generators of $\cW^{\text{free}}(\gs\gl_7)^{\mathbb{Z}_2}$, $U^{3,7}_{0,6}$ and $U^{5,7}_{0,4}$ cannot appear because $W^7$ does not appear in the OPEs of $W^i$ for $i<7$. However, in $\cW^k(\gs\gl_7)^{\mathbb{Z}_2}$, we have five relations of the form
\begin{equation} \begin{split} 
& :U^{3,5}_{0,0} U^{3,5}_{0,0}: -: U^{3,3}_{0,0} U^{5,5}_{0,0}:\ = a_{11} U^{3,3}_{0,10} + a_{12} U^{5,5}_{0,6} + a_{13} U^{3,5}_{0,8} + a_{14} U^{3,7}_{0,6}+ a_{15} U^{5,7}_{0,4} + \cdots,
\\ & :U^{3,3}_{0,0} U^{3,5}_{2,0}: - :U^{3,3}_{2,0} U^{3,5}_{0,0}: \ = a_{21} U^{3,3}_{0,10} + a_{22} U^{5,5}_{0,6} + a_{23} U^{3,5}_{0,8} + a_{24} U^{3,7}_{0,6} + a_{25} U^{5,7}_{0,4} + \cdots,
\\ & :U^{3,3}_{0,0} U^{3,5}_{1,1}: - :U^{3,3}_{1,0} U^{3,3}_{1,0}: \ = a_{31} U^{3,3}_{0,10} + a_{32} U^{5,5}_{0,6} + a_{33} U^{3,5}_{0,8} + a_{34} U^{3,7}_{0,6} + a_{35} U^{5,7}_{0,4} + \cdots,
\\ & :U^{3,3}_{0,0} U^{3,3}_{3,1}: - :U^{3,3}_{3,0} U^{3,3}_{1,0}: \ = a_{41} U^{3,3}_{0,10} +a_{42} U^{5,5}_{0,6} + a_{43} U^{3,5}_{0,8} + a_{44} U^{3,7}_{0,6} + a_{45} U^{5,7}_{0,4} + \cdots,
\\ & :U^{3,3}_{0,0} U^{3,3}_{2,2}: - :U^{3,3}_{2,0} U^{3,3}_{2,0}: \ = a_{51} U^{3,3}_{0,10} + a_{52} U^{5,5}_{0,6} + a_{53} U^{3,5}_{0,8} + a_{54} U^{3,7}_{0,6} + a_{55} U^{5,7}_{0,4} + \cdots,
\end{split}\end{equation}
where the remaining terms do not depend on $U^{3,3}_{0,10}$, $U^{5,5}_{0,6}$, $U^{3,5}_{0,8}$,  $U^{3,7}_{0,6}$, and $U^{5,7}_{0,4}$. The matrix $[a_{ij}]$ can be computed explicitly by computer and it is nonsingular for generic values of $k$. Therefore suitable linear combinations of these relations will yield decoupling relations for $U^{3,7}_{0,6}$ and $U^{5,7}_{0,4}$. It is easy to check that $U^{7,7}_{0,2}$ cannot appear in any relation, so this is the only strong generator in weight $16$ that is needed.

In weight $17$, there are three generators from Corollary \ref{cor:longlistwsln}, namely $U^{3,5}_{0,9}$, $U^{3,7}_{0,7}$, $U^{5,7}_{0,5}$. Only two of these generators, namely $U^{3,5}_{0,9}$ and $U^{3,7}_{0,7}$, can be eliminated in $\cW^{\text{free}}(\gs\gl_7)$. However, in $\cW^k(\gs\gl_7)^{\mathbb{Z}_2}$, we have three relations
\begin{equation} \begin{split} & :U^{3,3}_{1,0} U^{5,5}_{0,0}: - :U^{3,5}_{1,0} U^{3,5}_{0,0}:\  = a_{11} U^{3,5}_{0,9} + a_{12} U^{3,7}_{0,7} + a_{13} U^{5,7}_{0,5} + \cdots, 
\\ & :U^{3,3}_{0,0} U^{3,7}_{1,0}: - :U^{3,3}_{1,0} U^{3,7}_{0,0}:\  = a_{21} U^{3,5}_{0,9} + a_{22} U^{3,7}_{0,7} + a_{23} U^{5,7}_{0,5} + \cdots, 
\\ & :U^{3,3}_{0,0} U^{3,5}_{3,0}: - :U^{3,3}_{3,0} U^{3,5}_{0,0}: \ = a_{31} U^{3,5}_{0,9} +  a_{32} U^{3,7}_{0,7} + a_{33} U^{5,7}_{0,5} + \cdots, \end{split} \end{equation}
 Again, it is straightforward to compute the matrix $[a_{ij}]$ and check that it is invertible for generic $k$. Therefore all the generators in weight $17$ can be eliminated.

Finally, in weight $18$, there are six generators from Corollary \ref{cor:longlistwsln}, namely $U^{3,3}_{0,12}$, $U^{3,5}_{0,10}$, $U^{3,7}_{0,8}$, $U^{5,5}_{0,8}$, $U^{5,7}_{0,6}$, and $U^{7,7}_{0,4}$. Five of these generators, namely $U^{3,3}_{0,12}$, $U^{3,5}_{0,10}$, $U^{3,7}_{0,8}$, $U^{5,5}_{0,8}$, and $U^{5,7}_{0,6}$, can be eliminated in $\cW^{\text{free}}(\gs\gl_7)^{\mathbb{Z}_2}$, but there is no relation allowing $U^{7,7}_{0,4}$ to be decoupled. However, in $\cW^k(\gs\gl_7)^{\mathbb{Z}_2}$ we can decouple $U^{7,7}_{0,4}$ by applying the operator $(W^4)_{(1)}$ to the weight $16$ decoupling relation for $U^{5,7}_{0,4}$, since $$(W^4)_{(1)} U^{5,7}_{0,4} = a_{4,5} U^{7,7}_{0,4} + \cdots.$$ Here we are using that $(W^4)_{(1)} W^7$ is a normally ordered polynomial in $\{L,W^i|\ 3\leq i \leq 7\}$ and their derivatives, so the remaining terms appearing in $(W^4)_{(1)} U^{5,7}_{0,4}$ do not depend on $U^{7,7}_{0,4}$. Also, it is apparent that when $(W^4)_{(1)}$ is applied to the right hand side of the decoupling relation for $U^{5,7}_{0,4}$, the term $U^{7,7}_{0,4}$ cannot appear. \end{proof}

\section{The structure of $\cW(c,\lambda)^{\mathbb{Z}_2}$}
As in Section \ref{sec:winf}, we use the generating set $\{L, W^i|\ i \geq 3\}$ given in \cite{LII}, so that $W^3_{(5)} W^3 = \frac{c}{3} 1$ and $W^i = (W^3)_{(1)} W^{i-1}$ for $i\geq 4$.

\begin{lemma} \label{lem:weakwinf} After a suitable localization of the ring $\mathbb{C}[c,\lambda]$, $\cW(c,\lambda)^{\mathbb{Z}_2}$ has a weak generating set consisting of the union of the following sets:
\begin{enumerate}
\item $\{L, W^{2i}|\ i \geq 2 \}$,
\item $\{U^{3,2i+1}_{0,0}|\ i\geq 1\}$.
\end{enumerate}
\end{lemma}

\begin{proof} For each $d\geq 0$, let $\cW(c,\lambda)^{\mathbb{Z}_2}[d]$ denote the subspace of weight $d$, and let $V[d]\subseteq \cW(c,\lambda)^{\mathbb{Z}_2}[d]$ denote the span of all fields that can be written as linear combinations of words in $\{L, W^{2i}, U^{3,2i}_{0,0}\}$ and the $k^{\text{th}}$ vertex algebra products for $k\in \mathbb{Z}$. Then $V = \bigoplus_{d\geq 0} V[d]$ is exactly the subalgebra which is weakly generated by the above fields.

If $V \neq \cW(c,\lambda)^{\mathbb{Z}_2}$, there is some weight $d$ where any basis $\cW(c,\lambda)^{\mathbb{Z}_2}[d]$ contains a vector that is not in $V[d]$. Then the image of this vector would be needed in a basis of $\cW^k(\gs\gl_n)^{\mathbb{Z}_2} = \cC^{\psi}(0,n)$ for all but finitely many values of $n$. This contradicts Corollary \ref{cor:longlistweakwslnfree}. \end{proof}

\begin{thm} \label{lem:strongwinf} After a suitable localization of $\mathbb{C}[c,\lambda]$, $\cW(c,\lambda)^{\mathbb{Z}_2}$ has a strong generating set consisting of the following fields
\begin{enumerate}
\item $W^{2i}$, for $i \geq 1$,
\item $U^{3,3}_{0,2a}$, for $a = 0,1,2,3$,
\item $U^{3,2i+1}_{0,a}$, for all $i\geq 2$ and $a = 0,1,2,3,4,5,6$,
\item $U^{2i+1,2i+1}_{0,2a}$, for all $i \geq 2$ and $a = 0,1,2$,
\item $U^{2i+1,2j+1}_{0,a}$, for all $2 \leq i < j$ and $a = 0,1,2,3,4,5$.
\end{enumerate}
\end{thm}

\begin{proof} The proof is the same as the proof of Lemma \ref{lem:weakwinf}, using Corollary \ref{cor:strongwn} instead of Corollary \ref{cor:longlistweakwslnfree}. \end{proof}

\begin{cor} After a suitable localization of the ring $\mathbb{C}[\psi]$, $\cC^{\psi}(n,m)^{\mathbb{Z}_2}$ has a strong finite generating set
\begin{enumerate}
\item $L, W^{2i}$, for $ 2 \leq i \leq  \frac{(m+1)(m+n+1)-1}{2}$,
\item $U^{3,3}_{0,2a}$, for $a = 0,1,2,3$,
\item $U^{3,2i+1}_{0,a}$, for all $2 \leq i <  \frac{(m+1)(m+n+1)-1}{2} $ and $a = 0,1,2,3,4,5,6$,
\item $U^{2i+1,2i+1}_{0,2a}$, for all $2 \leq i <  \frac{(m+1)(m+n+1)-1}{2}$ and $a = 0,1,2$,
\item $U^{2i+1,2j+1}_{0,a}$, for all $2 \leq i < j <  \frac{(m+1)(m+n+1)-1}{2}$ and $a = 0,1,2,3,4,5$.
\end{enumerate}
\end{cor}

The above strong generating sets for $\cW(c,\lambda)^{\mathbb{Z}_2}$ and $\cC^{\psi}(n,m)^{\mathbb{Z}_2}$ are not minimal. In fact, the decoupling relations in $\cW^k(\gs\gl_7)^{\mathbb{Z}_2}$ in weights $16$, $17$, and $18$ that we found in the proof of Theorem \ref{thm:sl7min}, are the specializations of relations in $\cW(c,\lambda)^{\mathbb{Z}_2}$. So the generators $U^{3,7}_{0,6}$, $U^{5,7}_{0,4}$, $U^{5,7}_{0,5}$, and $U^{7,7}_{0,4}$ can be eliminated in $\cW(c,\lambda)^{\mathbb{Z}_2}$ and in $\cC^{\psi}(n,m)^{\mathbb{Z}_2}$ for all $n,m$. It is an interesting problem to determine the minimal strong generating sets for $\cW(c,\lambda)^{\mathbb{Z}_2}$ as well as $\cC^{\psi}(n,m)^{\mathbb{Z}_2}$.

We are now able to prove the main result in this paper.
\begin{thm} \label{thm:wt4}
Let $R$ be the localization of $\mathbb{C}[c,\lambda]$ for which Lemma \ref{lem:weakwinf} holds, and let $R'$ be the localization of $R$ obtained by inverting $\lambda$ and $-8 + 22 \lambda + 5 \lambda c$, if necessary. Then as a vertex algebra over $R'$, $\cW(c,\lambda)^{\mathbb{Z}_2}$ is generated by $W^4$.
 \end{thm}   
 
\begin{proof} Let $\langle W^4 \rangle \subseteq \cW(c,\lambda)^{\mathbb{Z}_2}$ denote the vertex subalgebra generated by $W^4$. In view of Lemma \ref{lem:weakwinf}, we need to show that all the fields $W^{2i}$ and $W^{3,2i+1}_{0,0}$ lie in $\langle W^4 \rangle$ after inverting $\lambda$ and $-8 + 22 \lambda + 5 \lambda c$. First, we claim that $L \in \langle W^4 \rangle$ with no restrictions on $c,\lambda$, i.e., with no localization required. From Equation (A.3) of \cite{LII} we have
$$(W^4)_{(5)} W^4 = -\frac{4}{3}(-125 + 32 \lambda(2+c)) L,$$ so $L$ lies in $\langle W^4 \rangle$ as long as $-125 + 32 \lambda(2+c) \neq 0$. Note that this holds for $c = -2$ for any value of $\lambda$.

Suppose next that $c \neq -2$ and $\lambda =  \frac{125}{32 (2 + c)}$, so that $(W^4)_{(5)} W^4 = 0$ and $L$ cannot be obtained in this way. Using the OPE algebra of $\cW(c,\lambda)$ specialized along the curve $\lambda =  \frac{125}{32 (2 + c)}$, we have the following computations:
\begin{equation}\begin{split} & 
(W^4)_{(5)} \big((W^4)_{(3)} W^4\big) = \frac{1250 (43 + 9 c)}{c +2} L,
\\ &  (W^4)_{(4)}  \bigg((W^4)_{(4)} \big((W^4)_{(3)} W^4\big) \bigg) = \frac{200 (5831 + 1353 c)}{c+2} L.\end{split} \end{equation}
It follows that for all $c\neq -2$, $L\in \langle W^4 \rangle$ for all values of $\lambda$.

Next, we claim that as long as $\lambda \neq 0$ and $\lambda \neq \frac{8}{22+5c}$, both $W^6$ and $:W^3 W^3:$ lie in $\langle W^4 \rangle$. This follows from the following calculations:
\begin{equation} \begin{split} W^4_{(1)} W^4 & = \frac{4}{5} W^6 - \frac{288 \lambda}{5} :W^3 W^3: + \cdots,
\\ (W^4)_{(3)} (W^4_{(1)} W^4) &= -\frac{32}{5} (-32 + 115 \lambda + 26 \lambda c) W^6 + \frac{2304}{5} \lambda (8 + 5 \lambda + \lambda c) :W^3 W^3: + \cdots
\end{split}\end{equation} Here the remaining terms only depend on $L,W^4$ and their derivatives, and hence lie in $\langle W^4 \rangle$. We can solve for $W^6$ and $:W^3 W^3:$ separately as long as $$ \text{det}  \left[\begin{array}{rrrrrrrrr}
				 & \frac{4}{5} \qquad \qquad \qquad & &  - \frac{288 \lambda}{5}  \qquad\qquad \\
				 &  \ \ \ \ \ \  &  \\
				 & -\frac{32}{5} (-32 + 115 \lambda + 26 \lambda c)  & &  \frac{2304}{5} \lambda (8 + 5 \lambda + \lambda c)   \\
			\end{array} \right] = -\frac{9216}{5} \lambda (-8 + 22 \lambda + 5 \lambda c) \neq 0.$$ 
Note that $-8 + 22 \lambda + 5 \lambda c = 0$ is just the truncation curve for $\cW^k(\gs\gl_3)$. It is necessary to localize along this curve because the quotient $\cW^k(\gs\gl_3)^{\mathbb{Z}_2}$ of $\cW(c,\lambda)^{\mathbb{Z}_2}$ is not generated by a field in weight $4$.

Next, we have 
\begin{equation} \begin{split}(W^4)_{(1)} U^{3,3}_{0,0} & = 2 U^{3,5}_{0,0} + \frac{1}{3} (13 - 128 \lambda - 16 \lambda c) U^{3,3}_{0,2} - 8\lambda \partial^2 U^{3,3}_{0,0} -\frac{32 \lambda}{3} :L\partial^2 W^4: 
\\ & + \frac{112 \lambda}{3} :(\partial^2 L) W^4: + \frac{56 \lambda}{3} :(\partial L)(\partial W^4): +12 \lambda :( \partial^3 L) (\partial L): 
\\ & + \frac{32 \lambda}{3} :(\partial^4 L) L: + \frac{1}{6} \partial^2 W^6 + \frac{-11 + 40 \lambda + 2 \lambda c}{9}\partial^4 W^4 + \frac{149 - 688 \lambda + 40 \lambda c}{1080} \partial^6 L.
\end{split} 
\end{equation}
Since $U^{3,3}_{0,2}$, $U^{3,3}_{0,0}$, and $L$ can be generated from $W^4$, it follows that $U^{3,5}_{0,0} \in \langle W^4 \rangle$ with no further localization required.
Next, we have
$$(W^4)_{(1)} W^6 = \frac{4}{7} W^8 + \cdots,$$ where the remaining terms lie in $\langle W^4 \rangle$, so $W^8 \in \langle W^4 \rangle$.
We also have 
\begin{equation}
\begin{split} (W^4)_{(1)} W^8 & = \frac{4}{9} W^{10} -\frac{376 \lambda}{3} U^{5,5}_{0,0}  -\frac{752 \lambda}{3} U^{3,7}_{0,0} + \cdots,
\\ (W^4)_{(1)} U^{3,5}_{0,0} & =  U^{5,5}_{0,0} + \frac{2}{3} U^{3,7}_{0,0} + \cdots,
\\ (W^6)_{(1)} U^{3,3}_{0,0} & = 2 U^{3,7}_{0,0} + \cdots,\end{split} \end{equation} where the remaining terms lie in $\langle W^4 \rangle$, so $W^{10}$, $U^{5,5}_{0,0}$ and $U^{3,7}_{0,0}$ all lie in $\langle W^4 \rangle$ with no further localization required.

Let $V_{2a}$ denote the vector space spanned by $\{U^{2r+1, 2s+1}_{0,0}|\ 1\leq r \leq s,\ 2r + 2s + 2 = 2a\}$. We have an injective linear map
$$f: V_{2a} \rightarrow V_{2a+2},\qquad f(U^{2r+1, 2s+1}_{0,0} )= a_{4, 2r+1} U^{2r+3, 2s+1}_{0,0} + a_{4,2s+1} U^{2r+1, 2s+3}_{0,0},$$ which agrees with the restriction of $(W^4)_{(1)}$ to $V_{2a}$ up to terms which depend only on $W^{2j}$ for $j\leq a$ and $U^{2k+1, 2\ell+1}_{a,b}$ for $2k+2\ell + 2 \leq 2a$. Here the constants $a_{i,j}$ are given by Lemma \ref{formofab}. We now assume inductively that the fields $W^{2a}$ and the spaces $V_{2a}$ lie in $\langle W^4 \rangle$ for all $a \leq i+1$, with no further localization required. We have already checked the base cases where $i \leq 4$. Observe first that
$$(W^4)_{(1)} W^{2i+2} = a_{4,2i+2} W^{2i+4} + \cdots,$$ where the remaining terms either lie in $V_{2i+4}$, or in $\langle W^4 \rangle$ by inductive hypothesis. Since $W^{2i+4}$ cannot appear in $(W^4)_{(1)} U$ for any $U \in V_{2i+2}$, it suffices to show that $V_{2i+4}$ lies in $\langle W^4 \rangle$. Note that our inductive hypothesis also implies that $f$ agrees with the restriction of $(W^4)_{(1)}$ to $V_{2i+2}$ up to terms which lie in $\langle W^4 \rangle$.

If $i$ is even, $V_{2i+4}$ and $V_{2i+2}$ both have dimension $\frac{i}{2}$. Since $f: V_{2i+2} \rightarrow V_{2i+4}$ is injective, it is also surjective, so $V_{2i+4} \subseteq \langle W^4 \rangle$.

If $i$ is odd, $V_{2i+4}$ has dimension $\frac{i+1}{2}$ and $V_{2i+2}$ has dimension $\frac{i-1}{2}$, so the image $f(V_{2i+2}) \subseteq V_{2i+4}$ has codimension one. It is easy to check that $U^{i,i+4}_{0,0}$ spans a complement of $f(V_{2i+2})$. But
$$(W^6)_{(1)} U^{i, i}_{0,0} = 2 a_{6, i} U^{i,i+4}_{0,0} + \cdots,$$ where the remaining terms lie in $\langle W^4 \rangle$. Hence $U^{i,i+4}_{0,0}$ lies in $\langle W^4 \rangle$ as well. \end{proof}

\begin{cor} \label{cor:wt4}
After a suitable localization of $\mathbb{C}[\psi]$, $\cC^{\psi}(n,m)^{\mathbb{Z}_2}$ is generated as a vertex algebra by $W^4$ for all $n\geq 4$ and $m=0$, and all $n\geq 1$ if $m\geq 1$. In particular, this holds for $\cW^k(\gs\gl_4)^{\mathbb{Z}_2}$ for all $n\geq 4$.
 \end{cor}   

It is an interesting problem to determine the minimal localization $R'$ of $\mathbb{C}[c,\lambda]$ needed for Theorem \ref{thm:wt4} to hold, and similarly to determine the minimal localization of $\mathbb{C}[\psi]$ needed for Corollary \ref{cor:wt4} to hold. Suppose that $S$ is a localization of $\mathbb{C}[c,\lambda]$ for which $W^4$ generates $W^6$ and all fields $U^{3,3}_{0,2a}$ for $a\geq 0$. It is not difficult to show using the form of $a_{i,j}, b_{i,j}$ given by Lemma \ref{formofab} that no further localization is required, that is, we can take $R' = S$. Moreover, we expect that $S$ will be the localization of $\mathbb{C}[c,\lambda]$ along finitely many polynomials. This would imply that for each $m,n$, the localization needed for Corollary \ref{cor:wt4} to hold will require inverting finitely polynomials in $\psi$; equivalently, $W^4$ will generate $\cC^{\psi}(n,m)^{\mathbb{Z}_2}$ for all but finitely many values of $\psi$. We hope to return to this question later.


\begin{thebibliography}{ABKS}
\bibitem[AI]{AI} T. Arakawa, \textit{Associated varieties of modules over Kac-Moody algebras and $C_2$-cofiniteness of $\cW$-algebras}, Int. Math. Res. Not. IMRN 2015, no. 22, 11605-11666.
\bibitem[AII]{AII} T. Arakawa, \textit{Rationality of $\cW$-algebras: principal nilpotent cases}, Ann. Math. vol. 182, no. 2 (2015), 565-604.
\bibitem[ACL]{ACL} T. Arakawa, T. Creutzig, and A. Linshaw,  \textit{W-algebras as coset vertex algebras}, Invent. Math. 218, no. 1 (2019), 145-195.
\bibitem[AL]{AL} M. Al-Ali and A. Linshaw,  \textit{The $\mathbb{Z}_{2}$-orbifold of the $\cW_{3}$-algebra}, Comm. Math. Phys. 353, No. 3 (2017), 1129-1150.
\bibitem[ALY]{ALY} T. Arakawa, C. H. Lam, and H. Yamada, \textit{Parafermion vertex operator algebras and W-algebras}, Trans. Amer. Math. Soc. 371 (2019), no. 6, 4277-4301.
\bibitem[B-H]{B-H} R. Blumenhagen, W. Eholzer, A. Honecker, K. Hornfeck, and R. Hubel, \textit{Coset realizations of unifying W-algebras}, Int. J. Mod. Phys. Lett. A 10, 2367-2430 (1995)
\bibitem[BFM]{BFM}  M. Bershtein, B. Feigin, and G. Merzon, \textit{Plane partitions with a \lq\lq pit": generating functions and representation theory}, Selecta Math. (N.S.) 24 (2018), no. 1, 21-62. 
\bibitem[CLI]{CLI} T. Creutzig and A. Linshaw, \textit{The super $\cW_{1+\infty}$-algebra with integral central charge}, Trans. Am. Math.Soc. 367 (2015), 5521-5551.
\bibitem[CLII]{CLII}T. Creutzig and A. Linshaw, \textit{Cosets of affine vertex algebras inside larger structures}, J. Algebra 517 (2019), 396-438.
\bibitem[CLIII]{CLIII} T. Creutzig and A. Linshaw, \textit{Trialities of $\mathcal{W}$-algebras},  Cambridge J. Math. Vol 10, no. 1 (2022), 69-194.
\bibitem[FFI]{FFI} B. Feigin and E. Frenkel, \textit{Quantization of the Drinfeld-Sokolov reduction}, Phys. Lett. B, 246(1-2):75-81, 1990.
\bibitem[FFII]{FFII} B. Feigin and E. Frenkel, \textit{Duality in W-algebras}, Int. Math. Res. Notices 1991 (1991), no. 6, 75-82.
\bibitem[GR]{GR} D. Gaiotto and M. Rap\v{c}\'ak, \textit{Vertex Algebras at the Corner}, J. High Energy Phys. 2019, no. 1, 160, front matter+85 pp.
\bibitem[JWI]{JWI} C. Jiang and Q. Wang, \textit{Representations of $\mathbb{Z}_2$-orbifold of the parafermion vertex operator algebra $K(\mathfrak{sl}_2,k)$,} J. Algebra 529 (2019), 174-195.
\bibitem[JWII]{JWII} C. Jiang and Q. Wang, \textit{Fusion rules for $\mathbb{Z}_2$-orbifolds of affine and parafermion vertex operator algebras}, Israel J. Math. 240 (2020), no. 2, 837-887.
\bibitem[KRW]{KRW} V. Kac, S. Roan, and M. Wakimoto, \textit{Quantum reduction for affine superalgebras}, Comm. Math. Phys. 241 (2003), no. 2-3, 307-342.
\bibitem[KW]{KW} V. Kac and M. Wakimoto, \textit{Quantum reduction and representation theory of superconformal algebras}. Adv. Math. 185 (2004), no. 2, 400-458. 
\bibitem[KL]{KL} S. Kanade and A. Linshaw, \textit{Universal two-parameter even spin $\cW_{\infty}$-algebra}, Adv. Math. 355 (2019), 106774, 58 pp.
\bibitem[Li]{Li} H. Li, \textit{Vertex algebras and vertex Poisson algebras}, Commun. Contemp. Math. 6 (2004) 61-110.
\bibitem[LI]{LI} A. Linshaw, \textit{Invariant subalgebras of affine vertex algebras}, Adv. Math. 234 (2013), 61-84.
\bibitem[LII]{LII} A. Linshaw, \textit{Universal two-parameter $\mathcal{W}_{\infty}$-algebra and vertex algebras of type $\mathcal{W}(2,3,\dots,N)$}, Compos. Math. 157, no.1 (2021),12-82. 
\bibitem[RSYZ]{RSYZ} M. Rap\v{c}\'ak, Y. Soibelman, Y. Yang, and G. Zhao, \textit{Cohomological Hall algebras, vertex algebras, and instantons}, Comm. Math. Phys. 376 (2020), no. 3, 1803-1873.
\bibitem[Z]{Z} Y. C. Zhu, \textit{Modular invariants of characters of vertex operators}, J. Amer. Math. Soc. 9 (1996) 237-302.

\end{thebibliography}
\end{document}